  \pdfoutput=1
  \documentclass[microtype]{gtpart}

  \usepackage{graphicx}
  \usepackage[mathscr]{eucal}
  \usepackage{amssymb}
  \usepackage[dvipsnames]{xcolor}
  \usepackage{enumerate, cite}
  \usepackage[margin=1.4in]{geometry}
  \usepackage{hyperref}
  \usepackage[normalem]{ulem}
  \usepackage{inputenc}
  \newcommand{\br}{\mathbb{R}}
  
  \newcommand{\bz}{\mathbb Z}
  \newcommand{\bn}{\mathbb N}
  \newcommand{\bq}{\mathbb Q}

  \newcommand{\cE}{\mathcal E}

  \newcommand{\vp}{\varphi}

  \newcommand{\ssm}{\smallsetminus}

  \DeclareMathOperator{\mcg}{MCG}
  \DeclareMathOperator{\pmcg}{PMCG}

  \DeclareMathOperator{\Ends}{\cE}
  
  \DeclareMathOperator{\Homeo}{Homeo}

  \renewcommand{\co}{\colon\thinspace}

  \newtheorem{Thm}{Theorem}[section]
  \newtheorem{Thm*}{Theorem}
  \newtheorem{Prop}[Thm]{Proposition}
  \newtheorem{Lem}[Thm]{Lemma}
  \newtheorem{Cor}[Thm]{Corollary}
  \newtheorem{Cor*}[Thm*]{Corollary}
  
  \newtheorem{Question}[Thm]{Question}
  
  \newtheorem*{MainThm}{Theorem~\ref{thm:main}}
  \newtheorem*{MainThm2}{Theorem~\ref{thm:main2}}

  \theoremstyle{definition}
  \newtheorem{Def}[Thm]{Definition}
  
  \newtheorem*{Ex*}{Examples}
  \newtheorem{Rem}[Thm]{Remark}

  \numberwithin{equation}{section}

  \title{Non-planar ends are continuously unforgettable}
  
  \author{Javier Aramayona}
  \address{Instituto de Ciencias Matematicas (ICMAT) \\  280149, Madrid, Spain}
  \email{javier.aramayona@icmat.es}
  \author{Rodrigo de Pool}
  \address{Instituto de Ciencias Matematicas (ICMAT) \\ 280149, Madrid, Spain}
  \email{rodrigo.depool@icmat.es}
  
  \author[Skipper]{Rachel Skipper}
  \address{Department of Mathematics \\ University of Utah \\ Salt Lake City, Utah 84112, USA}
  \email{rachel.skipper@utah.edu}
  
  \author[Tao]{Jing Tao}
  \address{Department of Mathematics \\ University of Oklahoma \\ Norman, OK 73019, USA}
  \email{jing@ou.edu}
  
  \author[Vlamis]{Nicholas G.~Vlamis}
  \address{Department of Mathematics \\ CUNY Graduate Center \\ New York, NY 10016, USA and \newline Department of Mathematics \\ CUNY Queens College \\ Flushing, NY 11367, USA}
  \email{nicholas.vlamis@qc.cuny.edu}
  
  \author[Wu]{Xiaolei Wu}
  \address{Shanghai Center for Mathematical Sciences \\ Jiangwan Campus \\ Fudan University \\ Shanghai, 200438, P.R.~China}
  \email{xiaoleiwu@fudan.edu.cn}

\begin{document}  
  
\begin{abstract} 
We show that continuous epimorphisms between a class of subgroups of mapping class groups of orientable infinite-genus 2-manifolds with no planar ends are always induced by homeomorphisms. 
This class of subgroups includes the pure mapping class group, the closure of the compactly supported mapping classes, and the full mapping class group in the case that the underlying manifold has a finite number of ends or is perfectly self-similar. 
As a corollary, these groups are Hopfian topological groups. 
\end{abstract}

\maketitle

\vspace{-30pt}

%------------------------
% Introduction
%------------------------

\section{Introduction}

A fundamental tool in studying mapping class groups of finite-type surfaces is the forgetful homomorphism: Given a non-compact finite-type surface \( S \) and a point \( p \in S \), there is a surjective homomorphism \( \pmcg(S\ssm\{p\}) \to \pmcg(S) \) obtained by ``forgetting the puncture \( p \)'', where \( \pmcg(S) \) is the pure mapping class group.
For the forgetful homomorphism, the surface in the codomain has one less end than the surface in the domain.
The motivating question of this article is to ask if the same phenomenon can occur in the setting of infinite-genus surfaces without planar ends, or in other words, 
\begin{quote}
\emph{Can you forget a non-planar end?}
\end{quote}
This question appears as Problem 4.45 of \cite{AIM}.
In our main theorem, we give a negative answer to this question under a continuity assumption.

Let \( M \) be an orientable 2-manifold\footnote{We use the term 2-manifold to refer to a surface with empty boundary.}.
The mapping class group of \( M \), denoted \( \mcg(M) \), is the group of isotopy classes of orientation-preserving homeomorphisms \( M \to M \); the pure mapping class group, denoted \( \pmcg(M) \), is the kernel of the action of \( \mcg(M) \) on the end space \( \Ends(M) \) of \( M \). 
We equip \( \mcg(M) \) with the compact-open topology, that is, the quotient topology coming from the compact-open topology on \( \Homeo^+(M) \).  
With this topology, \( \mcg(M) \) is a topological group. 
Throughout the article, subgroups of \( \mcg(M) \) are assumed to be equipped with the subspace topology. 
We let \( \Gamma_c(M) \) denote the subgroup of \( \mcg(M) \) consisting of compactly supported mapping classes (i.e., the mapping classes with a representative that restricts to the identity outside of a compact set), and we let \( \Gamma(M) \) denote the closure of \( \Gamma_c(M) \). 

\begin{Def}
Let \( M \) be an infinite-type 2-manifold, and let \( H \) be a subgroup of \( \mcg(M) \).
\begin{itemize}
\item \( H \) is \emph{large} if it contains \( \Gamma(M) \). 
\item \( H \) is \emph{mostly pure} if \( H \cap \pmcg(M) \) has countable index in \( H \).
\end{itemize}
\end{Def}

For example, \( \pmcg(M) \) and \( \Gamma(M) \) are both large and mostly pure. 
On the other hand, \( \mcg(M) \) is large, but need not be mostly pure; in particular, \( \mcg(M) \) is mostly pure if and only if \( M \) has finitely many ends. 

Let \( M' \) be another orientable 2-manifold, and let \( G \) and \( G' \) be subgroups of \( \mcg(M) \) and \( \mcg(M') \), respectively. 
A homomorphism \( \varphi \co G \to G' \) is \emph{induced by a homeomorphism} if there exists a homeomorphism \( f \co M \to M' \) such that, for any homeomorphism \( g \co M \to M \) representing an element of G,  \( \varphi([g]) = [f \circ g \circ f^{-1}] \), where \( [g] \) denotes the isotopy class of \( g \). 

\begin{MainThm}
\label{thm:main}
Let \( M \) and \( M' \) be orientable 2-manifolds, and suppose \( M \) has infinite genus and has no planar ends.
If \( G \) is a mostly pure large subgroup of \( \mcg(M) \) and \( G' \) is a large subgroup of \( \mcg(M') \), then every continuous epimorphism \( G \to G' \) is induced by a homeomorphism.
\end{MainThm}

Here, it is worth stressing that, to the best of our knowledge, the analog of Theorem \ref{thm:main} is not known in the case of finite-type surfaces. In light of the forgetful and capping homomorphisms, a possible version of this problem is: 

\begin{Question}
    Let $M$ and $M'$ be finite-type
     surfaces of genus $g$ and $g'$, respectively, where $g\ge 3$ and $g'\ne g$.
   Are there any epimorphisms \( \pmcg(M) \to \pmcg(M') \)?
\end{Question}

Next, recall that an object \( A \) in a category is \emph{Hopfian} if every epimorphism \( A \to A \) is an automorphism. 
It is a standard fact that every finitely generated residually finite group is Hopfian (in the category of abstract groups).
As the mapping class group of a finite-type surface is residually finite (see \cite[\S6.4]{Primer}) and finitely generated (see \cite[\S4.3]{Primer}),  every finitely generated subgroup of the mapping class group of a finite-type orientable 2-manifold is Hopfian.  
One of the motivating problems for this article is to understand when \( \mcg(M) \) and \( \pmcg(M) \) are Hopfian groups. 
As a corollary of Theorem~\ref{thm:main}, we establish a class of big mapping class groups that are Hopfian in the category of topological groups, see Problem 4.5 of \cite{AIM}. 

\begin{Cor}
If \( M \) is an orientable infinite-genus 2-manifold with no planar ends, then every mostly pure large subgroup of \( \mcg(M) \) is a Hopfian topological group;
in particular, \( \pmcg(M) \) is a Hopfian topological group, and if \( M \) has finitely many ends, then so is \( \mcg(M) \).
\qed
\end{Cor}

Our proof of the main theorem relies crucially on continuity in several locations. 
In the final section of this article, we give an accounting of how continuity is used and an a priori weaker condition that suffices for our arguments. 
This weaker condition is motivated by the definition of \emph{infinitely multiplicative homomorphisms} given by Cannon--Conner \cite{CannonCombinatorial}.
We leave as a question whether continuity can be removed all together from the hypotheses in Theorem~\ref{thm:main}. 

\begin{Question}
If \( M \) is a non-compact orientable 2-manifold with no planar ends, is every mostly pure large subgroup of \( \mcg(M) \) a Hopfian  group?
\end{Question}

When discussing forgetting ends, it is natural to work in the setting of pure mapping class groups, but we expect Theorem~\ref{thm:main} to hold more generally for large subgroups. 
In the specific case of generalizing Theorem~\ref{thm:main} to \( \mcg(M) \), we can reduce the problem to asking if there exists an epimorphism \( \Homeo(\Ends(M)) \to G' \), where \( \Ends(M) \) denotes the end space of \( M \). 
We have an approach, based on the work of Afton--Calegari--Chen--Lyman \cite{AftonNielsen}, for answering this question in the negative, but it is not clear if it can be applied to all cases.
In the next theorem, we provide a class of 2-manifolds for which this approach is successful. 

A 2-manifold \( M \) is \emph{perfectly self-similar} if \( M \# M \) is homeomorphic to \( M \) and for every proper compact subset \( K \) of \( M \) there exists \( f \in \Homeo(M) \) such that \( f(K) \cap K = \varnothing \). 

\begin{MainThm2}
Let \( M \) and \( M' \) be orientable 2-manifolds, and suppose \( M \) is infinite genus with no planar ends. 
If \( M \) is perfectly self-similar and \( G' \) is a large subgroup of \( \mcg(M') \), then every continuous epimorphism from \( \mcg(M) \) to \( G' \) is induced by a homeomorphism. 
\end{MainThm2}

In Section~\ref{sec:full}, in addition to proving Theorem~\ref{thm:main2}, we extend the theorem to a larger subclass of self-similar 2-manifolds using \cite{BhatAlgebraic}, as well as give an example of a 2-manifold for which the proof techniques do not apply. 

As a corollary to Theorem~\ref{thm:main2}, we have the Hopfian property. 

\begin{Cor}
If \( M \) is a perfectly self-similar infinite-genus 2-manifold with no planar ends, then \( \mcg(M) \) is a Hopfian topological group. 
\qed
\end{Cor}

In the special case that  \( M \) is as in the corollary and the end space of \( M \) is a Cantor set (i.e., \( M \) is the blooming Cantor tree surface), the fifth author \cite{VlamisHomeomorphism} showed that \( \mcg(M) \) has the automatic continuity property, which implies that every endomorphism of \( M \) is continuous.
Therefore,  we can upgrade the corollary from the topological category to the abstract category.

\begin{Cor}
The mapping class group of the blooming Cantor tree surface is Hopfian. 
\qed
\end{Cor}

\subsubsection*{Motivation and context}

Motivated by Margulis's Superrigidity for lattices in Lie groups, Ivanov \cite{IvanovAutomorphisms} proved that automorphisms of finite-type mapping class groups are induced by surface homeomorphisms. Since then, there has been a significant amount of work dedicated to classifying homomorphisms between mapping class groups; we refer the reader to the survey article \cite{AS-survey} and the references therein for an overview.

In light of this, a natural question to ask is whether the same types of phenomena occur in the context of big mapping class groups. In this direction, Bavard--Dowdall--Rafi \cite{BavardIsomorphisms} proved the analog of Ivanov's theorem for big mapping class groups, showing that any isomorphism between finite-index subgroups of big mapping class groups is induced by a homeomorphism (we remark that previous partial results in this direction were shown by Patel and the fifth author in \cite{PatelAlgebraic}). On the other hand, Leininger, McLeay and the first author \cite{ALM-cohopf} constructed infinite families of (pure) big mapping class groups that are not {\em co-Hopfian} in the category of topological groups, that is, they admit continuous injective homomorphisms that are not surjective. On the other hand, subject to certain topological restrictions on the surfaces, the same authors proved that if a continuous injective homomorphism sends Dehn twists to Dehn twists, then it is induced by a proper surface embedding \cite[Theorem 3]{ALM-cohopf}.

For more general types of homomorphisms, the situation is even wilder.   
For instance, if \( M \) is an orientable infinite-genus 2-manifold with no planar ends, then \( \pmcg(M) \) admits non-trivial endomorphisms:
If \( M \) has at least two ends, then by a result of the first author, Patel, and the fifth author \cite{AramayonaFirst}, \( \pmcg(M) \) surjects onto \( \bz \), and hence, composing this surjection with a homomorphism \( \bz \to \pmcg(M) \) gives a nontrivial endomorphism.
In the case \( M \) has one end, i.e., \( M \) is the Loch Ness monster surface,  Domat--Dickmann \cite{DomatBig} showed that \( \mcg(M) \) surjects onto \( \bq \). 
Aougab, Patel, and the fifth author \cite{AougabIsometry} showed that \( \mcg(M) \) contains a copy of every countable group; in particular, it contains a copy of \( \bq \), and hence taking the composition \( \mcg(M) \to \bq \to \mcg(M) \) yields a nontrivial endomorphism. 

Finally, it is also not difficult to see that the non-planar hypothesis is necessary.  
If \( M \) is a 2-manifold with infinitely many isolated planar ends, then forgetting a single isolated planar end yields a continuous epimorphism \( \pmcg(M) \to \pmcg(M) \) that is not an isomorphism.  
For the full mapping class group, consider the 2-manifold \( M \) obtained by removing a copy of the ordinal space \( \omega^\omega+1 \) from the 2-sphere. 
The derived set of \( \omega^\omega+ 1 \)  is homeomorphic to \( \omega^\omega +1 \); therefore, forgetting all the isolated ends of \( M \) yields a continuous  epimorphism \( \mcg(M) \to  \mcg(M) \) that is not an isomorphism. 
These examples show that \( \mcg(M) \) and \( \pmcg(M) \) need not be Hopfian when \( M \) has planar ends.
However, despite these examples failing to be induced by a homeomorphism, they do arise from geometric constructions, i.e., by embedding the manifold into itself, which leads to the natural question:

\begin{Question}
Does every (continuous) epimorphism between (pure) mapping class groups come from a geometric construction?
\end{Question}

\subsubsection*{Acknowledgements} 

The authors thank Greg Conner for a helpful discussion about infinitely multiplicative homomorphisms, and we thank the anonymous referee for their careful reading and helpful suggestions, which undoubtedly improved the exposition. 

This article stems from discussions during the {\it Thematic Program on Geometric Group Theory and Low-Dimensional Topology} held at ICMAT during May--July 2024, which was part of ICMAT's {\it Agol Laboratory}. The authors thank ICMAT for its hospitality and acknowledge the support
from the Spanish Ministry of Science, Innovation, and Universities, through the Severo
Ochoa Programme for Centres of Excellence in R\&D (CEX2019-000904-S and CEX2023-
001347-S).

Aramayona was funded by grant  PGC2018-101179-B-I00 from the Spanish Ministry for Science, Innovation and Universities. 
De Pool was partially supported by the grant CEX2019-000904-S funded by MCIN/AEI/ 10.13039/501100011033 and from the grant PGC2018-101179-B-I00. 
Skipper was partially funded by NSF DMS-2005297 and the European Research Council (ERC) under the European Union’s Horizon 2020 research and innovation program (grant agreement No.725773).
Tao was partially funded by NSF DMS-1651963. 
Vlamis was partially supported by NSF  DMS-2212922 and PSC-CUNY awards 65331-00 53, 67380-00 54, and 67380-00 55. 
Wu was partially supported by NSFC No.12326601.

%------------------
% Preliminaries 
%-------------------

\section{Preliminaries} 
\label{sec:prelim}

\subsection{Surfaces and curves}

All \emph{surfaces} are assumed to be connected, Hausdorff, and second countable, and we will reserve the term \emph{2-manifold} to refer to a surface without boundary. 
A surface is \emph{planar} if it can be embedded in \( \br^2 \). 
A subsurface of a surface $S$ is a closed subset of $S$ that is a surface with respect to the subspace topology. We will usually consider subsurfaces up to isotopy. 
The \emph{space of ends} of $S$ is \[ \Ends(S) =  \lim_{\leftarrow} \pi_0(S\smallsetminus K), \] where the inverse limit is taken over the collection of all compact subsets $K \subset S$. 
Equipping $\pi_0(S \smallsetminus K)$ with the discrete topology, the limit topology on $\Ends(S)$ is totally disconnected, separable, and metrizable. 
An element $e \in \Ends(S)$ is called an \emph{end} of $S$.
An end is \emph{planar} if it is the end of a planar subsurface of \( S \).

A 2-manifold is of \emph{finite type} if it is homeomorphic to the interior of a compact surface (equivalently, its fundamental group is finitely generated); otherwise, it is of \emph{infinite type}. 
The classification of surfaces (see \cite{RichardsClassification}) says that the homeomorphism type of an orientable 2-manifold is determined by the following data: its genus (possibly infinite) and, up to homeomorphism, its space of ends together with the closed subset of its set of non-planar ends.

A simple closed curve on a surface $S$ is \emph{essential} if it does not bound a disk or a once-punctured disk and it is not homotopic to a boundary component.
We will use the term \emph{curve} to refer to the isotopy class of an essential simple closed curve. 
Given two curves $a$ and $b$, their \emph{geometric intersection number}, denoted \( i(a,b) \), is defined to be \( \min_{\alpha,\beta} |\alpha \cap \beta| \), where $\alpha$ and $\beta$ are representatives of $a$ and $b$, respectively. Two curves are \emph{disjoint} if $i(a,b)=0$. 

A \emph{multicurve} is a set of pairwise-disjoint curves. 
Two multicurves \( A \) and \( B \) are \emph{disjoint} if \( A \cap B = \varnothing \) and \( i(a,b) = 0 \) for all \( a \in A \) and for all \( b \in B \). 
A set of curves \( A \) is \emph{locally finite} if, given a curve \( b \), all but finitely many curves in \( A \) are disjoint from \( b \).

A \emph{pants decomposition} of a surface \( S \) is a locally finite multicurve \( P \) such that each component of $S \ssm P$ is homeomorphic to the thrice-punctured sphere (i.e., a \emph{pair of pants}). 
A surface admits a pants decomposition if and only if its fundamental group is non-abelian and every boundary component (if there are any) is compact. 

On several occasions, we will have to consider arcs in addition to curves.
We will only discuss arcs in the context of 2-manifolds, as opposed to general surfaces.
A \emph{simple proper arc} in a 2-manifold \( M \) refers to the image of a proper embedding \( \br \to M \); it is \emph{essential} if it does not bound an open disk in \( M \). 
We will use the term \emph{arc} to refer to the isotopy class of an essential simple proper arc. 
The definition of the geometric intersection number extends to the setting of arcs and curves. 

\subsection{Mapping class groups and their topology}

Given a surface $S$, the \emph{mapping class group}  of $S$ is the group $\mcg(S,\partial S)$ of isotopy classes of orientation-preserving homeomorphisms of $S$ fixing $\partial S$ pointwise, with isotopies also fixing $\partial S$ pointwise. 
If $\partial S = \varnothing$, we will shorten the notation to $\mcg(S)$.
The group $\mcg(S,\partial S)$ naturally inherits the quotient topology coming from the compact-open topology on the homeomorphism group of $S$. 

It will be useful to have a more workable description of this topology; this involves using a corollary to the Alexander method (see \cite[\S2.3]{Primer} for finite-type surfaces and see \cite{HernandezAlexander} or \cite{ShapiroAlexander} for infinite-type surfaces). 

\begin{Prop}
\label{prop:alexander-method}
Let \( M \) be a 2-manifold of genus at least three, and let \( f \in \mcg(M) \).
If  \( f(c) = c \) for every curve \( c \) in \( M \), then \( f \) is the identity. 
\qed
\end{Prop}

The genus restriction above is to rule out the finite number of 2-manifolds for which the statement fails.
This is exactly when the mapping class group has nontrivial center, which is also why it is stated for 2-manifolds rather than surfaces. 

The Alexander method implies that \( \mcg(M) \) acts faithfully on the set of curves in \( M \), and therefore, a mapping class is completely determined by the permutation it induces on the set of curves. 
Given a set of curves \( A \), set \[ U_A = \{ f \in \mcg(M) : f(a) = a \text{ for all } a \in A \}. \]
Then the \( U_A \) form a neighborhood basis for identity in the compact-open topology.
It follows that the sets of the form \( f \, U_A = \{ g \in \mcg(M) : g(a) = f(a) \text{ for all } a \in A \} \) form a basis for the compact-open topology (see \cite[\S4]{AramayonaBig}). 
A similar statement can be made if \( \partial S \neq \varnothing \) by allowing the set \( A \) to include arcs.

\subsection{Twists}

We will usually denote the (left) Dehn twist about a curve $a$ by $t_a$. 
If \( a \) is non-separating, then we say \( t_a \) is a non-separating Dehn twist. 
We record two standard facts about Dehn twists. 
Their proofs in \cite{Primer}  are in the setting of finite-type surfaces but readily extend to general surfaces. 

\begin{Lem}[\!\!{\cite[Fact~3.7]{Primer}}]
\label{lem:fact37}
Let \( a \) be a curve on a surface \( S \).
If \( f \in \mcg(S, \partial S) \), then \( f \, t_a \, f^{-1} = t_{f(a)} \). 
\qed
\end{Lem}

\begin{Lem}[\!{\cite[Fact~3.9]{Primer}}]
Two curves \( a \) and \( b \) in an orientable surface \( S \) are disjoint if and only if \( t_a(b) = b \). 
\qed
\end{Lem}

It is an exercise in the compact-open topology to show that if \( A \) is a locally finite multicurve containing infinitely many curves, \( \{ k_a\}_{a\in A} \) is a sequence of integers, and \( \{a_n\}_{n\in\bn} \) is an enumeration of the elements of \( A \), then \[ \lim_{n \to \infty} \prod_{j=1}^n t_{a_n}^{k_{a_n}} \] exists and does not depend on the choice of enumeration;  we write this limit as  \( \prod_{a\in A} t_a^{k_a} \). 
For a locally finite multicurve $A$, a \emph{multitwist} along $A$ is a mapping class of the form \( \prod_{a\in A} t_a^{k_a} \), where \( k_a \in \bz \). 
For simplicity, we will usually write $T_A$ for a multitwist along $A$.
We say a mapping class $h\in \mcg(\Sigma)$ is a \emph{root of a multitwist} if $h^m$ is a multitwist for some integer $m$.

\begin{Lem} 
\label{lem infinite twists}
Let \( A = \{ a_n\}_{n\in\bn} \) be a multicurve on a surface \( S \), and let \( \{ k_n \}_{n\in \bn} \) be a sequence of non-zero integers. 
Then \( \lim_{n\to\infty} \prod_{j=1}^n t_{a_n}^{k_n} \) exists if and only if $A$ is locally finite.
\end{Lem}

\begin{proof}
We already established the backwards direction above, so we only need to show the forwards direction. 
We will argue the contrapositive: let \( A \) fail to be locally finite. 
Then there exists a curve \( c \) and a subsequence of \( \{ a_{n_j} \} \) such that \( i(c, a_{n_j}) \neq 0 \) for all \( j \in \bn \). 
Let \( T_m =  \prod_{j=1}^{n_m} t_{a_{j}}^{k_{j}} \), and let \( F_m = T_m^{-1} T_{m+1} \). 
Note that if \( \lim_{n\to\infty} \prod_{j=1}^n t_{a_n}^{k_n} \) exists, then the sequence \( \{ T_m \}_{m\in \bn} \) converges, and hence \( \{ F_m \}_{m\in \bn} \) converges to the identity. 
We will show that the sequence \( \{ F_m \} \) does not converge to the identity, which will finish the proof. 

Observe that \( F_m(c) = t_{a_{n_{m+1}}}^{k_{n_{m+1}}}(c) \).
%For each \( m \in \bn \), choose a curve \( b_m \) satisfying \( i(c,b_m) = 0 \) and \( i(a_{n_{m+1}}, b_m) > 0 \). 
Then, by \cite[Proposition~3.2]{Primer}, 
\[
i(F_m(c), c) = i\left(t_{a_{n_{m+1}}}^{k_{n_{m+1}}}(c), c\right) = |k_{n_{m+1}}| i(a_{n_{m+1}}, c)^2
\]
%\[
%i(F_m(c), b_m) = i(t_{a_{n_{m+1}}}^{k_{n_{m+1}}}(c), b_m) \geq |k_{n_{m+1}}|i(a_{n_{m+1}}, c)i(a_{n_{m+1}}, b_m) - i(c, b_m),
%\]
and hence \( i(F_m(c), c) > 0 \).
It follows that  \( F_m(c) \neq c \) for all \( m \in \bn \), and in particular, there exists a neighborhood  of the identity in \( \mcg(S) \) that the sequence \( \{ F_m \}_{m\in \bn} \) never enters, namely \( \{ f \in \mcg(S) : f(c) = c \} \). 
Therefore,  the sequence \( \{ F_m\} \) does not converge to the identity. 
\end{proof}

From the above lemmas, we can readily deduce the following, using a limiting argument in the case of infinite multitwists: 

\begin{Lem} \label{lem:disjoint-commute}
Let \( S \) be an orientable surface, and let $T_A, T_B\in \mcg(S,\partial S)$ be non-trivial multitwists along the locally finite multicurves $A$ and $B$, respectively. Then $T_A$ and $T_B$ commute if and only if $i(a,b)=0$ for every $a\in A$ and $b\in B$. 
\qed
\end{Lem}

%------------------------
% Generating the compactly support mapping class group
%------------------------

\section{Generating compactly supported mapping classes}

A mapping class is \emph{compactly supported} if it has a representative that restricts to the identity outside a compact set. 
Given a 2-manifold \( M \), let \( \Gamma_c(M) \) be the subgroup of \( \mcg(M) \) consisting of compactly supported mapping classes.
In the proof of our main theorem, we will need to have a generating set for \( \Gamma_c(M) \) that is analogous to the Humphries generating set in the finite-type setting.
We separate this goal into its own section as it may be of independent interest. 

Let us begin by recalling the Humphries generating set (see \cite[\S4.4.3 and \S4.4.4]{Primer}). 
Humphries's theorem tells us that the Dehn twists about the curves shown in Figure~\ref{fig:humphries} generate \( \mcg(S, \partial S) \); in particular, for a compact surface with nonempty boundary, the Humphries generating set consists of \( 2g+b \) non-separating Dehn twists (the figure shows a genus three surface with three boundary components, but it should be clear how to generalize to higher genus and more boundary components).
We note that Humphries's theorem requires the surface to be orientable and nonplanar. 
We will build on Humphries's theorem to construct a generating set for \( \Gamma_c(M) \) when \( M \) is an orientable infinite-genus 2-manifold without planar ends.

\begin{figure}
\centering
\includegraphics[scale=0.75]{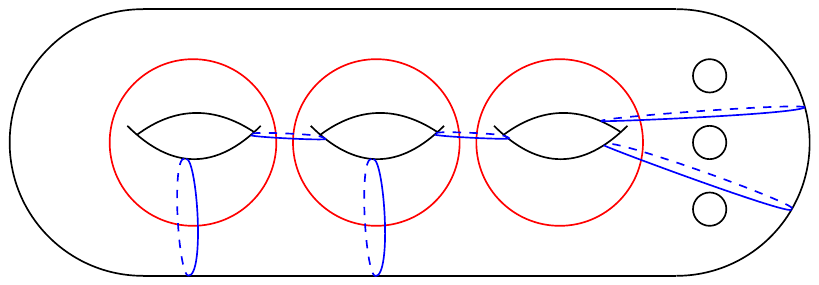}
\caption{The Dehn twists about the curves shown generate \( \pmcg(S) \), where \( S \) is the surface shown.  Here, \( S \) is a genus three surface with three boundary components.}
\label{fig:humphries}
\end{figure}

\begin{Lem}
\label{lem:exhaustion}
Let \( M \) be an orientable infinite-genus 2-manifold.
If \( M \) has no planar ends, then there exists a sequence of compact subsurfaces \( \{ \Sigma_n\}_{n\in\bn \cup\{0\}} \) in \( M \) such that 
\begin{enumerate}[(i)]
\item  \( \Sigma_0 \) is a genus one surface with one boundary component,
\item each component of \( \partial \Sigma_n \) is separating, 
\item \( \Sigma_n \subset \Sigma_{n+1} \), 
\item \( M = \bigcup_{n\in\bn} \Sigma_n \), and 
\item the closure of \( \Sigma_{n+1} \ssm \Sigma_n \) is a genus one surface with either  two or three boundary components.
\end{enumerate}
\end{Lem}

\begin{proof}[Sketch of proof]
Let \( T \) be a rooted tree whose end space is homeomorphic to the end space of \( M \) and such that the valence at the root is one but at all other vertices is either two or three (such a tree can be obtained by appropriately pruning an infinite binary tree). 
For \( i \in \{2,3\} \), let \( F_i \) be a genus one surface with \( i \) boundary components.  
We can then  build a surface using \( T \) as a blueprint as follows: if \( v \) is the root of \( T \), let \( F_v \) be a genus one surface with one boundary component, and for every other vertex \( v \) in \( T \), let \( F_v \) be a copy of \( F_i \), where \( i \) is the valence of \( v \). 
Take the disjoint union of all the \( F_v \) and glue them along their boundaries according to the adjacency relation in \( T \). 
By the classification of surfaces, the resulting surface is homeomorphic to \( M \).
The vertices of a binary tree have a canonical enumeration (in base two), inducing an order on its vertices. 
As \( T \) is obtained from pruning a binary tree, its vertices inherit an enumeration \( \{ v_i \}_{i \in \bn} \), labeled by the order in which they appear.
The surfaces of the form \( \Sigma_n = \bigcup_{i=0}^n F_{v_i} \) yield the desired sequence. 
\end{proof}

A \emph{chain} of curves in a surface is a locally finite set of curves such that any two curves in the set have geometric intersection number at most one; a chain is \emph{filling}  if every curve and arc on the surface has nontrivial geometric intersection with at least one curve in the chain. 
Associated to any chain \( \mathcal C \), we can define a graph \( \mathcal T(\mathcal C) \) as follows: the vertices are the curves in the chain, and two vertices are adjacent if their geometric intersection is one. 
If \( \mathcal T(\mathcal C) \) is a tree, we say the chain is \emph{tree-like}. 
An \emph{Alexander chain} is a filling tree-like chain consisting of non-separating curves.
Observe that the curves in Figure~\ref{fig:humphries} form an Alexander  chain.

In what follows, we will require a generating set for the mapping class group of a four-holed sphere.  
Let \( R \) be a compact four-holed sphere.  
If \( a \) and \( b \) are curves with \( i(a,b) = 2 \), then \( t_a \) and \( t_b \) together with the Dehn twists along any three of the four boundary components of \( R \) generate \( \mcg(R, \partial R) \) (this is an exercise combining the computation of generators for the mapping class group of a pair of pants \cite[\S3.6.4]{Primer} and a four-punctured sphere \cite[\S4.2.4]{Primer}, the inclusion homomorphism \cite[Proposition~3.18]{Primer}, and the lantern relation \cite[Proposition~5.1]{Primer}). 

Given a compact subsurface \( \Sigma \) of a 2-manifold \( M \) such that each component of \( \partial \Sigma \) is separating and essential, the inclusion of \( \Sigma \) into \( M \) induces a monomorphism \( \mcg(\Sigma, \partial \Sigma) \to \mcg(M) \), see \cite[Theorem 3.18]{Primer}. 
Abusing notation, we will identify \( \mcg(\Sigma, \partial \Sigma) \) with its image. 

\begin{figure}[t]
\centering
\includegraphics[scale=0.85]{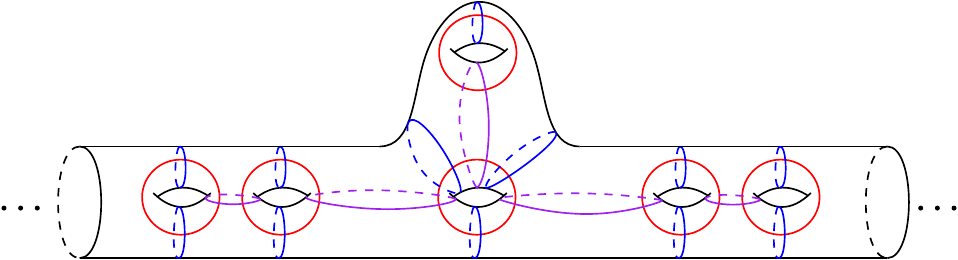}
\caption{The Alexander chain constructed in Theorem~\ref{thm:generators} on a 2-ended infinite-genus 2-manifold.}
\label{fig:alexander-chain}
\end{figure}

\begin{Thm}
\label{thm:generators}
Let \( M \) be an orientable infinite-genus 2-manifold.
If \( M \) has no planar ends, then there exists an Alexander chain \( \mathcal A \) in \( M \) such that the set of Dehn twists \( \{ t_a \}_{a\in \mathcal A} \) generates \( \Gamma_c(M) \).
Moreover, there exists an exhaustion of \( M \) by compact surfaces \( \{ \Sigma_n \}_{n\in\bn} \) such that each component of \( \partial \Sigma_n \) is separating and essential, such that \( \mathcal A_n = \{ a \in \mathcal A: a \subset \Sigma_n \} \) is a an Alexander chain in \( \Sigma_n \), and such that the set \( \{ t_a : a \in \mathcal A_n \} \) generates \( \mcg(\Sigma_n, \partial \Sigma_n) \). 
\end{Thm}

\begin{proof}
Let \( \{\Sigma_n\}_{n\in\bn} \) be the compact exhaustion of \( M \) given by Lemma~\ref{lem:exhaustion}. 
Observe that \[ \Gamma_c(M) = \bigcup_{n\in\bn} \mcg(\Sigma_n, \partial \Sigma_n). \]
We therefore need to give an Alexander chain \( \mathcal A \) such that the Dehn twists about the curves in \( \mathcal A \) contained in \( \Sigma_n \) generate \( \mcg(\Sigma_n, \partial \Sigma_n) \). 
Let us construct \( \mathcal A \) and then argue that it satisfies this property.

\begin{figure}
\centering
\includegraphics[scale=0.8]{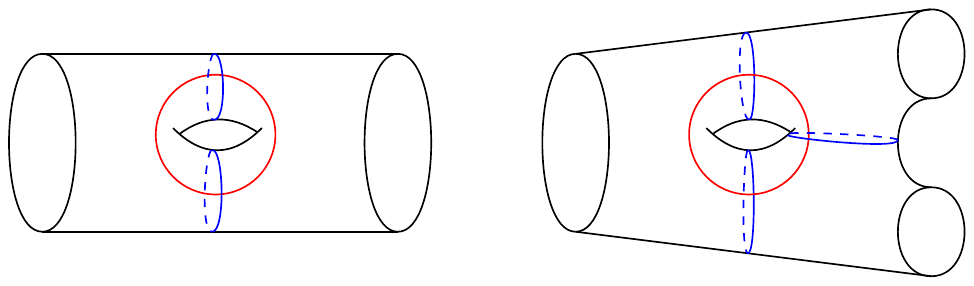}
\caption{The  surfaces \( F_2 \) (left) and \( F_3 \) (right) and their corresponding Alexander chains, \( C_2 \) and \( C_3 \).}
\label{fig:chain}
\end{figure}

Let \( S_0 = \Sigma_0 \), and for \( n \in \bn \), let \( S_n \) be the closure of \( \Sigma_{n} \ssm \Sigma_{n-1} \). 
Let \( \mathcal A_0 \) consist of any two simple closed curves with geometric intersection number one; hence, the curves in \( \mathcal A_0 \) generate \( \mcg(\Sigma_0, \partial \Sigma_0) \) (see \cite[\S3.6.4]{Primer}). 
For \( i \in \{2,3\} \), let \( F_i \) be the genus one compact surface with \( i \) boundary components.  
Let \( C_i \)  be the Alexander chain in \( F_i \) shown in Figure~\ref{fig:chain}. 
For \( n \in \bn \), let \( i(n) \in \{2,3\} \) such that \( S_n \) is homeomorphic to \( F_{i(n)} \), and fix a homeomorphism \( f_n \co F_{i(n)} \to S_n \). 
Let \( \mathcal A'\)  be the chain defined by \( \mathcal A' = \mathcal A_0 \cup \bigcup_{n = 1}^\infty f_n(C_{i(n)}) \).

We now add curves to the collection $\mathcal A'$ so that it forms an Alexander chain. 
Each of the curves in Figure~\ref{fig:chain} are colored either blue or red; each surface in the figure contains a single red curve that intersects each of the remaining curves---the blue curves---exactly once. 
Note that, in each surface, the blue curves give a pants decomposition.  
Designate one curve in \( \mathcal A_0 \) to be blue. 
Now, let \( \mathcal A'_b \) denote the subset of \( \mathcal A' \) containing all the blue curves. 
Each component of \( M \ssm \mathcal A'_b \) is a four-holed sphere. 
Enumerate the components as follows: define \( \Omega_n \) to be the unique component that is contained in \( \Sigma_n \) but not in \( \Sigma_{n-1} \).
Choose a non-separating curve \( a_n \) in \( M \) that is contained in \( \Omega_n \) and such that \( \mathcal A' \cup \{a_n\}  \) is a chain. 
Let \( \mathcal A = \mathcal A' \cup \{ a_n : n \in \bn\} \). 
By construction, \( \mathcal A \) is an Alexander chain. 
An example of the chain \( \mathcal A \) can be seen in Figure~\ref{fig:alexander-chain}.

For \( n \in \bn \cup\{0\} \), let \( \mathcal A_n \) be the subset of \( \mathcal A \) consisting of curves contained in \( \Sigma_n \). 
We claim that \( \{ t_a\}_{a\in \mathcal A_n} \) generates \( \mcg(\Sigma_n, \partial \Sigma_n) \). 
We proceed by induction. 
By construction, \( \{ t_a\}_{a \in \mathcal A_0} \) is a generating set for \( \mcg(\Sigma_0, \partial \Sigma_0) \). 
Now, assume that \( \{ t_a\}_{a \in \mathcal A_n} \) generates \( \mcg(\Sigma_n, \partial \Sigma_{n}) \). 
Recall that \( \Sigma_{n+1} = \Sigma_n \cup S_{n+1} \), where \( S_{n+1} \) is a compact genus one surface with either two or three boundary components.  

\begin{figure}[t]
\centering
\includegraphics{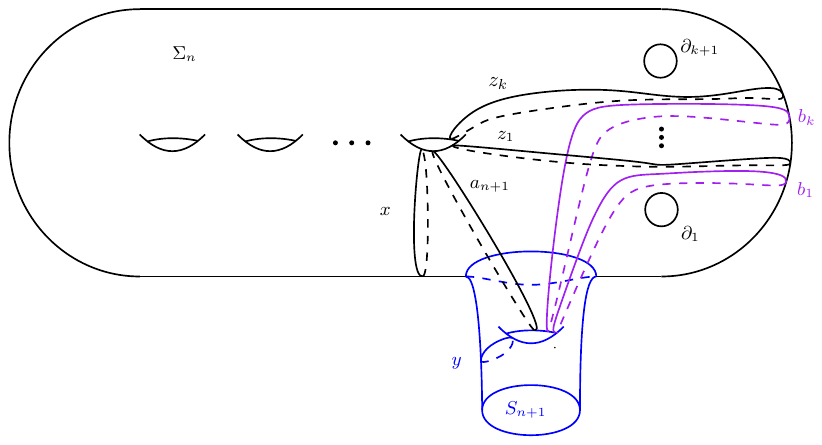}
\caption{The curves involved in the inductive step in the proof of Theorem~\ref{thm:generators}.}
\label{fig:induction}
\end{figure}

Let \( H \) be the subgroup generated by \( \{ t_a \}_{a \in \mathcal A_{n+1}} \).
In order to show that \( H = \mcg(\Sigma_{n+1}, \partial \Sigma_{n+1}) \), it is enough to show that \( t_{b_i} \in H \) for \( i \in \{1, \ldots, k \} \), where \( b_i \) is as shown in Figure~\ref{fig:induction}. 
As the set \( \{ t_a \}_{a \in \mathcal A_{n+1}} \cup \{ t_{b_i} \}_{i=1}^k \) contains a Humphries generating set for \( \mcg(\Sigma_{n+1}, \partial \Sigma_{n+1})\).  For the rest of the argument, we will use the curves as labelled in Figure~\ref{fig:induction}. 
The figure is drawn with \( S_{n+1} \) having two boundary components, but the same labeling and the following argument directly apply to the case where \( S_{n+1} \) has three boundary components.

For \( i \in \{1, \ldots, k\} \), consider the four-holed sphere \( R_i \) embedded in \( \Sigma_{n+1} \)  with boundary components \( x \), \( y \), \( b_i \), and \( z_i \). 
Observe that there exists an essential simple closed curve \( w_i \) in \( R_i \) that is contained in \( \Sigma_n \) and satisfies \( i(a_{n+1}, w_i) = 2 \). 
Then \( \{ t_{a_{n+1}}, t_{w_i}, t_x, t_y, t_{z_i} \} \) generates \( \mcg(R_i, \partial R_i) \).
Therefore, as \( t_{b_i} \in \mcg(R_i, \partial R_i) \) and \( \{ t_{a_{n+1}}, t_{w_i}, t_x, t_y, t_{z_i} \} \subset H \), we have that \( t_{b_i} \in H \), as desired.
\end{proof}

\begin{Rem}
Above we describe the Humphries generating set for an orientable non-planar compact surface; however, the Humphries theorem holds more generally for orientable non-planar surfaces of finite type, with \( \mcg(S, \partial S) \) replaced by \( \pmcg(S, \partial S) \). 
As a consequence, Theorem~\ref{thm:generators} can be readily extended to orientable infinite-genus 2-manifolds in which every planar end is isolated. 
\end{Rem}

%------------------------
% Irreducible homomorphisms
%------------------------

\section{Irreducible epimorphisms}

As \( \Gamma(M) \) is the smallest mostly pure large subgroup of \( \mcg(M) \), we will begin by working with  this group and later reduce the main theorem to this case.  

A homomorphism \( \varphi \co \Gamma(M) \to \Gamma(M') \) is \emph{compact}  if \( \vp(\Gamma_c(M)) \subset \Gamma_c(M') \); it is \emph{reducible} if there exists an arc or curve \( a \) in \( M' \) such that \( \varphi(g)(a) = a \) for all \( g \in \Gamma(M) \), and otherwise it is \emph{irreducible}. 
The goal of the section is to prove the following:

\begin{Thm}
\label{thm:compactly-supported}
Let \( M \) and \( M' \) be non-compact orientable 2-manifolds with no planar ends.
Every continuous compact irreducible homomorphism \( \Gamma(M) \to \Gamma(M') \) is induced by a homeomorphism. 
\end{Thm}

The key ingredient in the proof is to establish that such a homomorphism
sends Dehn twists to Dehn twists and preserves geometric intersection
numbers 0 and 1 between the associated curves.  From this, we can build the
desired homeomorphism. To accomplish this, we proceed with a sequence of
lemmas. As we proceed through the lemmas, we will put as few hypotheses on
each lemma as possible in hopes that they will be useful to others. 

Using our assumptions and \cite{BridsonSemisimple}, it is not hard to see
that $\varphi$ must map a Dehn twist $t_a$ about a curve $a$ to a
\emph{root} of a Dehn twist about a multicurve $\varphi_*(a)$
(Lemma~\ref{lem:Bridson}). We then eliminate roots via a sequence of steps,
taking advantage of the infinite-type setting. The first step
(Lemma~\ref{lem:disjoint1}) shows that if $a$ and $b$ are disjoint and
non-separating (individually and jointly) then $\varphi_*(a)$ and
$\varphi_*(b)$ are also \emph{disjoint} in a strong sense (see definition
below). This result relies on a pigeon-hole argument tailored to the
infinite-type setting, based on the existence of a locally finite,
\emph{non-separating}, infinite multicurve containing the curves $a$ and
$b$ (Lemma~\ref{lem:good_curves}). Next, a continuity argument using an
infinite-product trick shows that local finiteness of multicurves is
preserved under $\varphi_*$ (Lemma~\ref{lem:locally_finite}). This gives
rise to a second pigeon-hole argument that allows us to promote
disjointness to the supports of the roots themselves
(Lemma~\ref{lem:disjoint_supports}). Once the disjointness of supports is
established, we can use the lantern relation to eliminate roots
(Lemma~\ref{lem:multitwist}). Finally, we eliminate multiple components in
$\varphi_*(a)$ for a non-separating $a$
(Proposition~\ref{prop:twist2twist}), which also shows that intersection
number one is preserved under $\varphi$. This part of the argument relies
on the recent work of \cite{DePoolBraided} on the braid relation between
multitwists. From here, we move to the proof of
Theorem~\ref{thm:compactly-supported}, where we use the results of
\cite{AramayonaHomomorphisms} to obtain homeomorphisms between compact
exhaustions of \( M \) and \( M' \), allowing us to take a direct limit to
construct the desired homeomorphism \( M \to M' \). 

\begin{Lem}
\label{lem:no-torsion}
If \( M \) is an infinite-type orientable 2-manifold, then every nontrivial element of \( \Gamma_c(M) \) has infinite order.  
\end{Lem}

\begin{proof}
If \( g \in \Gamma_c(M) \) is nontrivial, then there exists a finite-type subsurface \( \Sigma \subset M \) containing the support of $g$ such that each boundary component of $\Sigma$ is an essential simple closed curve. 
In this case, the inclusion \( \Sigma \hookrightarrow M \) induces a monomorphism \( \mcg(\Sigma, \partial \Sigma) \to \Gamma_c(M) \) whose image contains \( g \). 
Therefore, we can view \( g \) as an element of \( \mcg(\Sigma, \partial \Sigma) \), which is torsion free (see \cite[Corollary~7.3]{Primer}). 
\end{proof}

Next, under the additional assumptions of non-triviality and continuity, we rule out the case that the image of a Dehn twist is trivial. 
As discussed above, pure mapping class groups of finite-type nonplanar surfaces are generated by Dehn twists about non-separating curves, and as all such Dehn twists are conjugate, these pure mapping class groups are normally generated by any such Dehn twist. 
It readily follows that \( \Gamma(M) \) is topologically normally generated by any non-separating Dehn twist. 

\begin{Lem}
\label{lem:infinite-order}
Let \( M \) be a non-planar orientable 2-manifold, let \( M' \) be an infinite-type orientable 2-manifold, and let \( t \in \Gamma(M) \) be a non-separating Dehn twist.
If \( \vp \co \Gamma(M) \to \Gamma(M') \) is continuous, compact, and nontrivial, then \( \vp(t) \) has infinite order. 
\end{Lem}

\begin{proof}
By Lemma~\ref{lem:no-torsion}, \( \vp(t) \) is either trivial or infinite order.
If \( \vp(t) \) is trivial, then \( \vp \) is trivial, as \( t \) is a topological normal generator of \( \Gamma(M) \) and \( \vp \) is continuous; hence, \( \vp(t) \) has infinite order.
\end{proof}

Next, we extend a result of Bridson \cite{BridsonSemisimple} to our setting. 

\begin{Lem}
\label{lem:Bridson}
Let \( M \) and \( M' \) be orientable 2-manifolds with the genus of \( M \) being at least three, and let \( t \in \mcg(M) \) be a Dehn twist.
If \( \vp \co \Gamma(M) \to \Gamma(M') \) is a compact homomorphism, then \( \vp(t) \) is the root of a multitwist. 
\end{Lem}

\begin{proof}
Let \( t \) be a Dehn twist about the curve \( a \) in \( M \). 
We can then find a compact surface \( \Sigma \) embedded in \( M \) such that \( a \subset \Sigma \), the inclusion of \( \Sigma \) into \( M \) induces a monomorphism \( \mcg(\Sigma, \partial \Sigma) \to \Gamma(M) \), and the genus of \( \Sigma \) is at least three. 
In particular, \( t \in \mcg(\Sigma, \partial \Sigma) \). 
As \( \vp \) is compact and as \( \mcg(\Sigma, \partial \Sigma) \) is finitely generated, there exists a compact surface \( \Sigma' \) of \( M' \) whose inclusion induces a monomorphism \( \mcg(\Sigma', \partial \Sigma') \to \Gamma(M') \) such that \( \mcg(\Sigma', \partial \Sigma') \) contains \( \vp(\mcg(\Sigma, \partial \Sigma)) \). 
We can then apply a result of Bridson \cite[Remark~3.1(3)]{BridsonSemisimple} to  \( \vp|_{ \mcg(\Sigma, \partial \Sigma)} \co \mcg(\Sigma, \partial \Sigma) \to \mcg(\Sigma', \partial \Sigma') \) to see that \( \vp(t) \) is a root of a multitwist. 
\end{proof}

Note that we allow for the identity to be considered a multitwist; however, Lemma~\ref{lem:infinite-order} and Lemma~\ref{lem:Bridson} imply that the image of a non-separating Dehn twist under a compact, continuous, nontrivial homomorphism is a root of a multitwist of infinite order.
We will now proceed to show that such an image is a nontrivial multitwist, i.e., it is not a root. 
But first, we introduce the following notation.

\textbf{Notation.} 
Let \( M \) and \( M' \) be orientable 2-manifolds with the genus of \( M \) being at least three, and let \( \vp \co \Gamma(M) \to \Gamma(M') \) be a continuous compact nontrivial homomorphism.
If \( a \) is a non-separating simple closed curve in \( M \), Lemma~\ref{lem:infinite-order} and Lemma~\ref{lem:Bridson} imply there exists a multicurve, which we denote \( \vp_*(a) \), such that \( \vp(t_a) \) is a root of a multitwist along \( \vp_*(a) \).

A multicurve is \emph{non-separating} if  the union of all its elements is non-separating; it is \emph{infinite} if it is an infinite set.
Note that every subset of a non-separating multicurve is itself a non-separating multicurve; in particular, each curve in a non-separating multicurve is non-separating. 
Let us record a readily verifiable topological fact.

\begin{Lem}
\label{lem:good_curves}
Let \( M \) be an infinite-genus orientable 2-manifold. 
If \( \{a,b\} \) is a non-separating multicurve, then there exists a locally finite non-separating infinite multicurve containing \( a \) and \( b \).  
\qed
\end{Lem}

We use this topological observation to show that the images of two Dehn twists about curves that together form a non-separating multicurve are multitwists along disjoint multicurves. Recall that two multicurves $A$ and $B$ are disjoint if they have no curves in common and $i(a,b)=0$ for all $a \in A$ and $b \in B$.
 
\begin{Lem}
\label{lem:disjoint1}
Let \( M \) and \( M' \) be orientable 2-manifolds, with \( M \) of infinite genus, and let \( \vp \co \Gamma(M) \to \Gamma(M') \) be a continuous compact nontrivial homomorphism. 
If \( \{a,b\} \) is a non-separating multicurve in \( M \), then \( \vp_*(a) \) and \( \vp_*(b) \) are disjoint. 
\end{Lem}

\begin{proof}
First observe that \( t_a \) and \( t_b \) commute, so \( \vp(t_a) \) and \( \vp(t_b) \) also commute, as do their powers.  
Since $\vp(t_a)$ and$\vp(t_b)$ are roots of multitwists along $\vp_*(a)$ and $\vp_*(b)$, it follows then, by Lemma \ref{lem:disjoint-commute}, that \( i(a',b') = 0 \) for all $a' \in \vp_*(a)$ and $b' \in \vp_*(b)$. 
So it is left to show that \( \vp_*(a) \cap \vp_*(b) = \varnothing \). 

By Lemma~\ref{lem:good_curves}, there exists a locally finite non-separating infinite multicurve \( \{a,b, c_1, c_2, \ldots \} \). 
Let \( A' = \vp_*(a) \), \( B' = \vp_*(b) \), and  \( C'_n = \vp_*(c_n) \) for each \( n \in \bn \). 
By the change of coordinates principle, for each \( n \in \bn \), there exists \( f_n \in \Gamma_c(M) \) such that \( f_n(a) = a \) and \( f_n(b) = c_n \). 
Let \( m \in \bn \) such that \( \vp(t_a^m) \) is a multitwist.
It follows that \( \vp(t^m) \) is a multitwist for any non-separating Dehn twist \( t \). 
As \( f_n \) commutes with \( t_a^m \) and conjugates \( t_b^m \) to \( t_{c_n}^m \),
it follows from Lemma~\ref{lem:fact37} that \( \vp(f_n)(A') = A' \) and \( \vp(f_n)(B') = C'_n \).
Moreover, 
\[ \vp(f_n)(A'\cap B') = \vp(f_n)(A') \cap \vp(f_n)(B') = A' \cap C'_n, \]
where the first equality uses that fact that \( \vp(f_n) \) gives a bijection on the set all curves of \( M' \). 

Let us assume that \( A' \cap B' \) is non-empty, so that \( A' \cap C'_n \) is non-empty for all \( n \in \bn \).
As the power set of \( A' \) is finite, by passing to a subsequence, we may assume that \( A'\cap C'_{k} = A' \cap C'_{j} \) for all \( j, k \in \bn \). 
This allows us to define \( X = \vp(f_{k})(A' \cap B') \) to be independent of \( k \).  
Now observe that given \( j\neq k \in \bn \), there exists \( g_{j,k} \in \Gamma(M) \) such that \( g_{j,k}(a) = c_{j} \) and \( g_{j,k}(b) = c_{k} \).
It follows that \( \vp(g_{j,k})(A' \cap B') = C'_{j} \cap C'_{k} \), and hence \( | A' \cap B'| = | C'_{j} \cap C'_{k} | \). 
As \( |X| = |A' \cap B'| \) and as \( X \subset C'_{j} \cap C'_{k} \), we have that \( C'_{j} \cap C'_{k} = X \).

Let \( t \) be a Dehn twist about a non-separating simple closed curve. 
Note that the set \( \{c_{k}\}_{k \in \bn } \) is locally finite, so \( t \) commutes with infinitely many of the \( c_{k} \), and hence \( \vp(t) \) must commute with any multitwist about \( X \), implying \( \vp(t)(X) = X \). 
As \( t \) was arbitrary, the previous statement holds for all non-separating Dehn twists, and as the non-separating Dehn twists generate \( \Gamma_c(M) \), we have that \( \vp(f)(X) = X \) for all \( f \in \Gamma_c(M) \).

For \( k \in \bn \), let \( t_k \) be the Dehn twist about \( c_{k} \).
Let \( m \in \bn \) be as above, so that \( \vp(t_k^m) \) is a multitwist, denoted \( T_k \), about the multicurve \( C'_{k} \). 
As \( X \subset C'_k \), there exist multitwists \( T_{X, k} \) and \( T'_k \) about \( X \) and \( C'_k \ssm X \), respectively, such that \( T_k = T_{X,k} T'_k \). 
Choose \( h_k \in \Gamma_c(M) \) so that \( h_k(c_{1}) = c_{k} \).
It follows that  \( \vp(h_k) \) conjugates \( T_1 \) to \( T_k \).
From the previous paragraph, we know that \( \vp(h_k) \) must fix \( X \) (setwise), and so \( \vp(h_k) \) conjugates \( T_{X,1} \) to \( T_{X,k} \). 
Since $X$ is finite, its permutation group is finite; hence, by passing to a subsequence,  we may assume that $\vp(h_{k})$ induces the same permutation on $X$ for all $k$. 
Therefore, we have \( T_{X,i} = T_{X,j} \) for all $i,j\in \bn$. 
This allows us to define \( T_{X} = T_{X,1} \) and to write 
\[
T_{k} = T_{X} T_{k}'
\]
for every \( k \in \bn \). 

To finish, for \( n \in \bn \), consider the mapping class \( F_n = \prod_{k =1}^n t_{k} \). 
Then \[ \vp(F_n) = \prod_{k=1}^n T_{k} = T^n_{X} \prod_{k=1}^n T_{k}' .\]
From the fact that the multicurve \( \{c_{k}\}_{k\in\bn} \) is locally finite, the limit of the \( F_n \) exists; however, the sequence \( \{\vp(F_n)\}_{n\in\bn} \) diverges.
This is not possible as \( \vp \) is continuous, and therefore, \(  A' \cap B' \) must be empty.
\end{proof}

From Lemma~\ref{lem:disjoint1}, we readily deduce the following.

\begin{Lem}
\label{lem:locally_finite}
Let \( M \) and \( M' \) be orientable 2-manifolds with \( M \) of infinite genus, and let \( \vp \co \Gamma(M) \to \Gamma(M') \) be a continuous compact nontrivial homomorphism. 
If \( A = \{ a_i \}_{i\in I} \) is a locally finite multicurve in \( M \) and each $a_i$ is a non-separating curve, then \( \vp_*(A) := \bigcup_{i\in I} \vp_*(a_i)  \) is a locally finite multicurve in \( M' \). 
\end{Lem}

\begin{proof}
It readily follows from Lemma~\ref{lem:disjoint1} that \( \vp_*(A) \) is a multicurve.
Choose \( m \in \bz \) such that \( \vp(t_a^m) \) is a multitwist for each \( a \in A \),
choose an enumeration \( \{a_n\}_{n\in \bn} \) of \( A \), and let \( T_n = \prod_{i=1}^n \vp(t_{a_n}^m) \).
Then, there exists an enumeration \(  \{ b_i \}_{i\in \bn} \) of \( \vp_*(A) \) such that, for each \( n \in \bn \),  there exist \( j_n \in \bn \) and \( \{ k_1, \ldots, k_{j_n} \} \subset \bz \ssm\{0\} \) satisfying \( \vp(T_n) = \prod_{i=1}^{j_n} t_{b_i}^{k_i} \).
By continuity, \( \lim \vp(T_n) \) exists, as  \( \lim t_{a_n}^m \) exists. 
By Lemma \ref{lem infinite twists}, the existence \( \lim \vp(T_n) \)  implies that \( \vp_*(A) \) is locally finite. 
\end{proof}

It is worth noting that we have not yet put any restriction on the 2-manifold in the codomain in any of the above lemmas, but the previous lemma implies that it must be of infinite type. 

Thus far, we have established that the image of two commuting non-separating Dehn twists have powers that are multitwists about disjoint multicurves; in other words, their images have powers with disjoint support. 
The next lemma upgrades this to saying that their images have disjoint supports.
Recall that the \emph{support} of a homeomorphism \( f\co M \to M \) is the closure of the set \( \{ x \in M : f(x) \neq x\} \), and we say that two mapping classes have \emph{disjoint supports} if they have representative homeomorphisms with disjoint supports. 

Given a locally finite multicurve \( A \) on a surface \( S \), we abuse notation and write \( S \ssm A \) to refer to the surface obtained by removing the union of representatives of the curves in \( A \) that are chosen to be pairwise disjoint (this can always be done by choosing a hyperbolic metric and taking geodesic representatives). 

\begin{Lem}
\label{lem:disjoint_supports}
Let \( M \) and \( M' \) be orientable 2-manifolds with \( M \) of infinite genus, and let \( \vp \co \Gamma(M) \to \Gamma(M') \) be a continuous compact nontrivial homomorphism. 
If \( \{a, b\} \) is a non-separating multicurve, then \( \vp(t_a) \) and \( \vp(t_b) \) have disjoint supports. 
\end{Lem}

\begin{proof}
Let \( A' = \vp_*(a) \), let \( B' = \vp_*(b) \), and let \( M'_a = M' \ssm A'  \). 
By Lemma~\ref{lem:disjoint1}, we know that \( A' \) and \( B' \) are disjoint multicurves. 
As \( A' \) is a finite multicurve, \( M_a' \) has finitely many components.
Let \( f \in \Gamma_c(M) \) commute with \( t_a \).
It follows that \( \vp(f)(A') = A' \), and hence permutes the components of \( M_a' \).
As \( \vp(f) \in \Gamma_c(M') \), it must be the case that \( \vp(f) \) fixes each infinite-type component of \( M_a' \). 

Let \( \Sigma \) be an infinite-type component of \( M'_a \). 
Letting \( \rm{Stab}(A') \) denote the stabilizer of \( A' \) in \( \Gamma(M') \), we see that there exists a restriction homomorphism \( r: \rm{Stab}(A') \to \Gamma(\Sigma) \) given by taking the restriction of an element of \( \rm{Stab}(A') \) to \( \Sigma \). 
Under this homomorphism, any multitwist about \( A' \) is trivial, and hence, \( (r\circ\vp)(t_a) \) has finite order in \( \Gamma(\Sigma) \), as \( \vp(t_a) \) is a root of a multitwist about \( A' \). 
But, \( (r\circ \vp)(t_a) \in \Gamma_c(\Sigma) \), and hence is trivial by Lemma~\ref{lem:no-torsion}, allowing us to conclude that \( \vp(t_a) \) has a representative that restricts to the identity on the interior of any infinite-type component of \( M_a' \). 

From the above, we observe that if each component of \( M_a' \) has infinite type, then \( \vp(t_a) \) is a multitwist.
As \( \vp(t_b) \) is conjugate to \( \vp(t_a) \), it is also a multitwist, and as \( A' \) and \( B' \) are disjoint multicurves, \( \vp(t_a) \) and \( \vp(t_b) \) have disjoint supports, as desired. 

The previous paragraph dealt with the case in which \( M_a' \) has no finite-type component, so we may now assume that \( M_a' \) contains at least one finite-type component.
Let \( F_a' \) be the closures in \( M' \) of the finite-type components of \( M_a' \). 
Then there exists a representative of \( \vp(t_a) \) supported in the union of \( F_a' \) with annuli about the curves in \( A' \). 

Let \( b' \in B' \). 
As \( A' \) and \( B' \) are disjoint, \( b' \) is either an essential curve in \( F_a' \) or in \( M \ssm F_a' \). 
Let \( C = \{a, b, c_1, c_2, \ldots \} \) be a locally finite non-separating multicurve, which is guaranteed to exist by Lemma~\ref{lem:good_curves}.
By Lemma~\ref{lem:locally_finite}, \( \vp_*(C) \) is a locally finite multicurve in \( M' \). 
So, there exists \( c_i \in C \) such that \( \vp_*(c_i) \) is disjoint from \( F_a' \). 
We can then choose \( f \in \Gamma_c(M) \) such that \( f(a) = a \) and \( f(b) = c_i \), so that \( \vp(f)(B') = \vp_*(c_i) \); in particular, \( \vp(f)(b') \subset M \ssm F_a' \). 
As \( f(a) = a \), we have that \( f \) commutes with \( t_a \). 
From above, it follows  that \( \vp(f)(F_a') = F_a' \).
Therefore, \( b' \) cannot be an essential curve in \( F_a' \), as \( \vp(f)(F_a') = F_a' \) and \( \vp(f)(b')  \subset M \ssm F_a' \), implying no curve in \( B' \) is contained in \( F_a' \).

Defining \( F_b' \) analogously, a similar argument shows that no curve in \( A' \) is contained in \( F_b' \). 
As \( A' \) and \( B' \) are disjoint multicurves, this forces \( F_a' \) and \( F_b' \) to be disjoint, and hence, the supports of \( \vp(t_a) \) and \( \vp(t_b) \) are disjoint. 
\end{proof}

We can now eliminate the possibility of roots. 

\begin{Lem}
\label{lem:multitwist}
Let \( M \) and \( M' \) be orientable 2-manifolds with \( M \) of infinite genus, and let \( \vp \co \Gamma(M) \to \Gamma(M') \) be a continuous compact nontrivial homomorphism. 
If \( t \in \Gamma(M) \) is a non-separating Dehn twist, then \( \vp(t) \) is a multitwist. 
\end{Lem}

\begin{proof}
Every surface of genus at least three contains an embedded 4-holed sphere \( \Sigma \) in which each of its boundary components and each curve contained in \( \Sigma \) is non-separating in the ambient surface.  
Let \( a_1 \), \( a_2 \), \( a_3 \), and \( a_4 \) be the boundary components of \( \Sigma \).
By the lantern relation (see \cite[Proposition~5.1]{Primer}), there exist curves \( a_5 \), \( a_6 \), and \( a_7 \) in \( \Sigma \) such that 
\[
t_1 = t_5t_6t_7 (t_2t_3t_4)^{-1},
\]
where \( t_i \) is the Dehn twist about \( a_i \) and such that \( \{a_1, a_i\} \) is a non-separating multicurve for each \( i \neq 1 \).

By Lemma~\ref{lem:disjoint_supports}, for each \( i \), we can choose a representative of \( \vp(t_i) \) in \( M ' \) supported in a subsurface \( \Sigma_i' \) such that \( \Sigma_1' \) and \( \Sigma_i' \) are disjoint whenever \( i \neq 1 \). 
In particular, by the lantern relation, \( \vp(t_1) \) has a representative supported in \( \bigcup_{i=2}^7 \Sigma_i' \), which is disjoint from \( \Sigma_1' \). 
The only way this is possible is if \( \vp(t_1) \) is supported in a union of annuli, implying \( \vp(t_1) \) is a multitwist. 
But, \( t_1 \) is conjugate to \( t \), and hence \( \vp(t) \) is a multitwist. 
\end{proof}

So far, we have established that the image of a Dehn twist is a multitwist under the given hypotheses.  
We now want to show that the image is in fact a Dehn twist.  
To accomplish this, we will need to use irreducibility and the non-planarity of the ends, which has not been the case up to this point.

First, we will need a result of the second author characterizing when multitwists are braided, as well as a quick lemma.
We say two multitwists \( T_1 \) and \( T_2 \) are \emph{braided} if they satisfy the braid relation, that is, if \( T_1T_2T_1 = T_2T_1T_2 \). 
If \( a \) and \( b \) are two simple closed curves, then \( i(a,b) = 1 \) if and only if \( t_a \) and \( t_b \) are braided (see \cite[Propositions~3.11 \& 3.13]{Primer}).  
The following theorem generalizes this fact to multitwists. 

\begin{Thm}[\!{\cite{DePoolBraided}}]
\label{thm:depool}
Let \( S \) be an orientable surface.
If \( T_1 \) and \( T_2 \) are braided multitwists, then there exist \( k \in \bn \), simple closed curves \( a_1, \ldots, a_k \), \( b_1, \ldots, b_k \), a multitwist $T$ fixing every curve $a_i, b_j$ for $i,j \in \{1, \ldots, k\}$, and \( n_i \in \{-1,1\} \) for \( i \in \{1, \ldots, k \} \) such that
\begin{align*}
T_1 &= T \prod_{i=1}^k t_{a_i}^{n_i} \\
T_2 &= T \prod_{i=1}^k t_{b_i}^{n_i}
\end{align*}
and such that for \( i,j \in \{1, \ldots, k\} \)
\[
i(a_i,b_j) = \left\{
    \begin{array}{l l}
    1 & \text{if } i = j \\
    0 & \text{if } i \neq j
    \end{array} \right..
\]
\qed
\end{Thm}

\begin{Lem}
\label{lem:cyclic}
Let \( M \) be a nonplanar orientable 2-manifold,  and let \( \vp \co \Gamma_c(M) \to H \) be a homomorphism  to a group $H$. 
If there exist curves \( a \) and \( b \) in \( M \) such that \( i(a,b) = 1 \) and \( \vp(t_a) = \vp(t_b) \), then \( \vp(\Gamma_c(M)) \) is cyclic.
Moreover,  if \( M \) has genus at least three, then \( \vp \) is trivial. 
\end{Lem}

\begin{proof}
If \( c \) and \( d \) are curves such that \( i(c,d) = 1 \), then there exists \( g \in \Gamma_c(M) \) such that \( g(a) = c \) and \( g(b) = d \). 
Therefore, \( \vp(t_c) = \vp(g\, t_a\,  g^{-1}) = \vp(g \, t_b \,  g^{-1}) = \vp(t_d) \). 
Now, given any two non-separating curves \( a' \) and \( b' \), there exist curves \( c_0, c_1, \ldots, c_k \) such that \( i(c_i, c_{i+1}) = 1 \), \( c_0 = a' \), and \( c_k = b' \) (see \cite[Lemma~4.5]{Primer}). 
It follows that \( \vp(t_{c_i})= \vp(t_{c_{i+1}}) \), and hence \( \vp(t_{a'}) = \vp(t_{b'}) \). 
As \( \Gamma_c(M) \) is generated by non-separating Dehn twists, we have that \( \vp(\Gamma_c(M)) \) is cyclic. 
Now, if the genus of \( M \) is at least three, then \( \Gamma_c(M) \) has trivial abelianization---this is an application of the lantern relation, see \cite[Section~5.1]{Primer} for a proof in the finite-type setting that extends to the infinite-type setting verbatim. 
Now, the image of \( \vp \) is abelian, as it is cyclic, and hence \( \vp \) is trivial. 
\end{proof}

We say two chains \( \mathcal C \) and \( \mathcal C' \) are isomorphic if there exists a bijection \( \psi \co \mathcal C \to \mathcal C' \) such that \( i(\psi(a), \psi(b)) = i(a,b) \) for all \( a, b \in \mathcal C \); we call \( \psi \) a \emph{chain isomorphism}. 

\begin{Prop}
\label{prop:twist2twist}
Let \( M \) and \( M' \) be orientable 2-manifolds with \( M \) of infinite genus and with no planar ends,  and let \( \vp \co \Gamma(M) \to \Gamma(M') \) be a continuous compact irreducible homomorphism. 
\begin{enumerate}[(1)]
\item If \( \mathcal A \) is an Alexander chain in \( M \) such that \( \{t_a\}_{a\in\mathcal A} \) generates \( \Gamma_c(M) \), then \( \mathcal A' = \bigcup_{a \in A} \vp_*(a) \) is an Alexander chain in \( M' \). 
Moreover, the map \( a \mapsto \vp_*(a) \) is a chain isomorphism from \( \mathcal A \) to \( \mathcal A' \).

\item If \( t \in \Gamma(M) \) is a non-separating Dehn twist, then \( \vp(t) \) is a non-separating Dehn twist. 
\end{enumerate}
\end{Prop}

\begin{proof}
Apply Theorem~\ref{thm:generators} to obtain an Alexander chain \( \mathcal A \) in \( M \) such that \( \{t_a\}_{a\in\mathcal A} \) generates \( \Gamma_c(M) \), and let \( \mathcal A' = \bigcup_{a \in A} \vp_*(a) \).
Note that, by Lemma~\ref{lem:multitwist}, the images $\varphi(t_a)$ of the \( t_a \) are multitwists. 

Let \( a, b \in \mathcal A \), let \( a' \in \vp_*(a) \), and let \( b' \in \vp_*(b) \). 
If \( i(a,b) = 0 \), then $t_a$ and $t_b$ commute, and  $\varphi(t_a)$ and $\varphi(t_b)$ are commuting multitwists. 
Therefore, by Lemma~\ref{lem:disjoint-commute}, \( i(a',b') = 0 \). 
We will use the contrapositive of this fact multiple times, so we record it here for reference:
\begin{equation}
\label{eq1}
\text{if }\, i(a',b') \neq 0 \text{ then   } \,i(a,b) = 1. \tag{\(\star\)}
\end{equation}
Now, if \( i(a,b) \neq 0 \), then \( i(a,b) = 1 \), implying \( t_a \) and \( t_b \) are braided. 
It follows that the multitwists $\varphi(t_a)$ and $\varphi(t_b)$ are braided; hence, by Theorem~\ref{thm:depool}, \( \vp_*(a) \cup \vp_*(b) \) is a chain, so  \( i(a',b') \in \{0,1\} \). 
In either case, we have that \( i(a',b') \in \{0,1\} \).

To show that \( \mathcal A' \) is a chain, it is left to show that it is locally finite.
As \( \mathcal A \) is given by the construction in Theorem~\ref{thm:generators}, we can write \( \mathcal A =  A_1 \cup  A_2 \cup  A_3 \), where \( A_1 \), \( A_2 \), and \( A_3 \) are pairwise disjoint locally finite multicurves (e.g., take \( A_1 \) to be the red curves, \( A_2 \) to be the blue curves, and \( A_3 \) to be  the remaining curves, which are purple in Figure~\ref{fig:alexander-chain}). 
By Lemma~\ref{lem:locally_finite}, \( \vp_*(A_i) \) is locally finite; hence, \( \mathcal A' = \vp_*(A_1) \cup \vp_*(A_2) \cup \vp_*(A_3) \) expresses \( \mathcal A' \) as a finite union of locally finite sets, implying \( \mathcal A' \) is itself locally finite. 
Therefore, \( \mathcal A' \) is a chain.

Recall that \( \mathcal T(\mathcal A') \) is the graph whose vertices are the curves of \( \mathcal A' \) and that two vertices are adjacent if they have geometric intersection one; to show that \( \mathcal A' \) is tree-like, we must show that \( \mathcal T(\mathcal A') \) is a tree, i.e., that it is connected and has no cycles.
Suppose \( a_1', a_2', \ldots, a_k' \) are distinct curves  \( \mathcal A' \)  forming a cycle in \( \mathcal T( \mathcal A' ) \).
Let \( a_1, \ldots, a_k \in \mathcal A \) such that \( a_i' \in \vp_*(a_i) \). 
Reading indices modulo \( k \), we have by \eqref{eq1} that \( i(a_i,a_{i+1}) = 1 \).
We claim that \( a_i \), \( a_{i+1} \), and \( a_{i+2} \) are distinct. 
To see this, observe that as \( \vp(t_{a_i}) \) and \( \vp(t_{a_{i+1}}) \) are braided, Theorem~\ref{thm:depool} implies that each component of \( \vp_*(t_{a_{i+1}}) \) has nontrivial geometric intersection with exactly one component of \( \vp_*(t_{a_i}) \).
In particular, as \( i(a_i', a_{i+1}') = 1 \) and \( i(a_{i+1}', a_{i+2}') = 1 \), we can conclude that \( a_{i+2}' \notin \vp_*(a_i) \), allowing us to conclude that \( a_i \), \( a_{i+1} \), and \( a_{i+2} \) are distinct curves. 
Therefore, if \( \mathcal T (\mathcal A') \) contains a cycle, so does \( \mathcal T ( \mathcal A ) \). 
It follows that \( \mathcal T(\mathcal A') \) cannot have any cycles, as \( \mathcal T(\mathcal A) \) is a tree. Finally, \( \mathcal T(\mathcal A') \) is connected as \(\mathcal A'\) is filling.
This establishes that \( \mathcal T(\mathcal A') \) is a tree and that \( \mathcal A' \) is an Alexander chain.

Next, we claim that if \( a,b \in \mathcal A \) such that \( i(a,b) = 1 \), then \( \vp_*(a) \) and \( \vp_*(b) \) are non-separating curves that have geometric intersection one. 
First, we show that \( \vp_*(a) \) is a curve for each \( a \in \mathcal A \), or equivalently, that \( |\vp_*(a)| = 1 \).
Suppose \( |\vp_*(a)| > 1 \), and let \( a', a'' \in \vp_*(a) \). 
As \( \mathcal T ( \mathcal A' ) \) is connected, we can choose a path \( a' = c_0', \ldots, c_m' = a'' \) in \( \mathcal T( \mathcal A') \) connecting \( a' \) and \( a'' \). 
Let \( c_i \in \mathcal A \) such that \( c_i' \in \vp_*(c_i) \). 
Arguing as we did above, \( c_0, \ldots, c_m \) is a path in \( \mathcal T(\mathcal A) \) that must contain a cycle, contradicting the fact that \( \mathcal T ( \mathcal A) \) is a tree. 
Therefore, \( |\vp_*(a)| = 1 \), and \( \vp_*(a) \) is a curve.

Now, let \( a, b \in \mathcal A \) such that \( i(a,b) = 1 \). 
Suppose that \( i(\vp_*(a), \vp_*(b)) = 0 \). 
Then \( \vp(t_a) \) and \( \vp(t_b) \) commute and are braided.  
It is a quick exercise to show that this implies that \( \vp_*(t_a) = \vp_*(t_b) \). 
By Lemma~\ref{lem:cyclic},  \( \vp(\Gamma_c(M)) < \ker \vp \), and hence by continuity, \( \vp \) is trivial.
But this is not the case, as  \( \vp \) is irreducible. 
Therefore, \( i(\vp_*(a), \vp_*(b)) = 1 \). 
At this point, we have shown that \( \vp_* \co \mathcal A \to \mathcal A' \) is a bijection and preserves intersection number; hence, it is  a chain isomorphism. 

Now, for any $a \in \mathcal{A}$, there exists $b\in \mathcal{A}$ such that $i(a,b)=1$. 
Hence, from above, we may conclude that \( \vp_*(a) \) and \( \vp_*(b) \) are curves and \( i(\vp_*(a), \vp_*(b)) = 1 \), implying that \( \vp_*(a) \) and \( \vp_*(b) \) are non-separating curves.
It follows then that both $\varphi(t_a)$ and \( \varphi(t_b) \)  are  powers of non-separating Dehn twists.
Moreover, as \( \vp(t_a) \) and \( \vp(t_b) \) are braided, Theorem~\ref{thm:depool} implies that they are in fact each non-separating Dehn twists.
As any two non-separating Dehn twists are conjugate, $\varphi(t)$ is a non-separating Dehn twist for every non-separating Dehn twist $t \in \Gamma(M)$. 
This completes the proof.
\end{proof}

The idea for the final step is to realize that if \( M \) and \( M' \) are infinite-genus 2-manifolds  admitting isomorphic Alexander chains, then \( M \) and \( M' \) are homeomorphic and the isomorphism between the chains is induced by a homeomorphism between the manifolds.  
To do this, we will realize the homeomorphism as a direct limit of homeomorphisms between compacts sets in an exhaustion, which we will obtain by appealing to the work of the first author and Souto \cite{AramayonaHomomorphisms}. 

Given a chain \( \mathcal C \) in a 2-manifold \( M \), fix a hyperbolic metric on \( M \), and for each \( c \in \mathcal C \), let \( \gamma_c \) be its geodesic representative. 
We call \( \bigcup_{c\in \mathcal C} \gamma_c \) a \emph{realization of the chain}.

\begin{proof}[Proof of Theorem~\ref{thm:compactly-supported}]
By Theorem~\ref{thm:generators}, we can choose a Alexander chain \( \mathcal A \) such that \( \{ t_a \}_{a\in \mathcal A} \) generates \( \Gamma_c(M) \). 
Theorem~\ref{thm:generators} also provides a sequence of compact subsurfaces \( \{ \Sigma_n \}_{n\in\bn} \) of \( M \) such that 
\begin{enumerate}[(i)]
\item each component of \( \partial \Sigma_n \) is separating and essential,
\item \( \Sigma_n \subset \Sigma_{n+1} \), 
\item \( M = \bigcup_{n\in\bn} \Sigma_n \), and
\item  \( \mathcal A_n = \{ a \in \mathcal A : a \subset \Sigma_n \} \) is an Alexander chain in \( \Sigma_n \), and 
\item \( \{ t_a : a \in \mathcal A_n \} \) generates \( \mcg(\Sigma_n, \partial \Sigma_n) \).
\end{enumerate}
By possibly forgetting the first few surfaces in the sequence, we  may assume that the genus of \( \Sigma_n \) is at least six for all \( n \in \bn \). 

Let \( \mathcal A' = \vp_*(\mathcal A) \), and for \( n \in \bn \), let \( \mathcal A_n' = \vp_*(\mathcal A_n) \).
By Proposition~\ref{prop:twist2twist}, \( \mathcal A' \) is an Alexander chain, and \( \vp_* \) induces an isomorphism between \( \mathcal A \) and \( \mathcal A' \), as well as \( \mathcal A_n \) and \( \mathcal A_n' \). 
Let \( \Sigma_n' \) be the subsurface of \( M' \) filled by the chain \( \mathcal A_n' \), that is, \( \Sigma_n' \) is obtained by taking a regular neighborhood of a realization of  \( \mathcal A_n' \) and then taking the union of this neighborhood with each component of its complement that is a disk. 
Note that \( \Sigma_n \) is the surface filled by \( \mathcal A_n \). 

Given a pair of simple closed curves that have geometric intersection one, a regular neighborhood of their union is a once-punctured torus.
This motivates the following definition: the \emph{lower genus} of a finite chain is the cardinality of a maximal collection of pairs of curves in the chain in which the curves in each pair intersect once and the curves in distinct pairs are pairwise disjoint.
It follows that the genus of a regular neighborhood of a realization of a finite chain is bounded below by the lower genus of the chain.

By the construction of \( \Sigma_n \), the lower genus of \( \mathcal A_n \) agrees with the genus of \( \Sigma_n \). 
Moreover, there are pairs \( (a_1, b_1), \ldots, (a_g, b_g) \in \mathcal A_n^2 \) realizing the lower genus of \( \mathcal A_n \) such that any two distinct curves in \( \mathcal A_n \ssm \{ a_1, b_1, \ldots, a_g, b_g\} \) are disjoint (e.g., take the \( a_i \) to be the red curves from Figure~\ref{fig:chain}). 
As the lower genus of \( \mathcal A_n' \) is equal to that of \( \mathcal A_n \), the genus of \( \Sigma_n' \) is at least the genus of \( \Sigma_n \). 
Now, suppose that the genus of \( \Sigma_n' \) is greater than the genus of \( \Sigma_n \).
Let \( a_i' = \vp_*(a_i) \) and \( b_i' = \vp_*(b_i) \), so  \( (a_1',b_1'), \ldots, (a_g', b_g') \) is a maximal collection of pairs of curves in \( \mathcal A_n' \) realizing its lower genus. 
Let \( F_i' \) be a closed regular neighborhood of a realization of \( \{a_i',b_i'\} \), so that \( F_i' \) is a one-holed torus.
We can assume that \( F_i' \cap F_j'  = \varnothing \) if \( i \neq j \). 
Then \( \Sigma_n' \ssm \bigcup F_i' \) has positive genus, implying that there exists a pair of curves in \( \mathcal A_n' \ssm \{ a_1', b_1', \ldots, a_g', b_g'\} \) having nontrivial geometric intersection, a contradiction.
Therefore, \( \Sigma_n \) and \( \Sigma_n' \) have the same genus.

Given the setup at hand, for each \( n \in \bn \), the restriction of \( \varphi \) to \( \mcg(\Sigma_n, \partial \Sigma_n) \) induces a homomorphism \( \vp_n \co \mcg(\Sigma_n, \partial \Sigma_n) \to \mcg(\Sigma_n', \partial \Sigma_n') \). 
We claim that \( \vp_n \) is an . 
As \( \Sigma_n \) and \( \Sigma_n' \) have the same genus (and it is at least six), we can apply a result of the first author and Souto \cite[Theorem~1.1]{AramayonaHomomorphisms} to see that \( \vp_n \) is induced by an embedding \( h_n \co \Sigma_n \to \Sigma_n' \).

As \( \vp_m \) restricts to \( \vp_n \) on \( \mcg(\Sigma_n, \partial \Sigma_n) \) for \( m > n \), we can choose the embeddings such that \( h_m \) restricts to \( h_n \) on \( \Sigma_n \).
We now claim that \( h_n \) can be chosen to be a homeomorphism. 
As \( \Sigma_n' \) is filled by curves contained in \( h_n(\Sigma_n) \), each component of \( \Sigma_n' \ssm h_n(\Sigma_n) \) is either a disk or annulus. 
We claim that each such component is a boundary annulus, i.e., is homotopic to a component of \( \partial \Sigma_n' \).
This implies that \( h_n \) can be chosen to be a homeomorphism. 
Let \( m \in \bn \) such that \( m > n \) and each component of \( \Sigma_m \ssm \Sigma_n \) has positive genus. 
It follows that each component of the complement of \( \Sigma_m' \ssm h_m(\partial \Sigma_n) \) has positive genus; in particular, if \( \delta_1 \) and \( \delta_2 \) are any two components of \( \partial \Sigma_n \), then \( h_m(\delta_1) \) does not bound a disk and \( h_m(\delta_1) \) and \( h_m(\delta_2) \) do not co-bound an annulus. 
The claim now follows, that is, each component of \( \Sigma_n' \ssm h_n(\Sigma_n) \) is a boundary annulus, and \( h_n \) can be chosen to be a homeomorphism. 
Therefore, \( \vp_n \co \mcg(\Sigma_n, \partial \Sigma_n) \to \mcg(\Sigma_n', \partial \Sigma_n') \) is an isomorphism.

Taking the direct limit of the \( h_n \), we obtain an embedding \( h \co M \to M' \). 
As \( \mathcal A' \) is an Alexander chain, we have that \( M' = \bigcup_{n\in\bn} \Sigma_n' \), and hence \( h \) is a homeomorphism. 
Let \( h_* \co \Gamma(M) \to \Gamma(M') \) be the isomorphism induced by \( h \). 
By continuity, as \( h_* \) and \( \vp \) agree on \( \Gamma_c(M) \), they must agree everywhere, so \( \vp = h_* \). 
\end{proof}

%--------------------
% General case
%---------------------

\section{Proof of the main theorem}

The goal of this section is to prove our main theorem by showing that it follows from Theorem~\ref{thm:compactly-supported}. 
Let us recall the statement of the main theorem that we aim to prove.

\begin{Thm}
\label{thm:main}
Let \( M \) and \( M' \) be orientable 2-manifolds and suppose \( M \) has infinite genus and no planar ends.
If \( G \) is a mostly pure large subgroup of \( \mcg(M) \) and \( G' \) is a large subgroup of \( \mcg(M') \), then every continuous epimorphism \( G \to G' \) is induced by a homeomorphism.
\end{Thm}

Given a continuous epimorphism \( \vp \co G \to G' \) with \( G \) and \( G' \) as in the theorem, we will see that the surjectivity will allow us to conclude that \( \vp \) restricted to \( \Gamma(M) \) induces a continuous compact irreducible homomorphism \( \Gamma(M) \to \Gamma'(M) \), which will allow us to apply Theorem~\ref{thm:compactly-supported}. 
We will take the same tact to prove Theorem~\ref{thm:main2} below. 
As the proof of these two theorems share the same core idea, we proceed with a sequence of lemmas/propositions that will be required for both theorems. 
The first required proposition is a result from Bavard--Dowdall--Rafi \cite{BavardIsomorphisms}.

\begin{Prop}[{\cite[Proposition~4.2]{BavardIsomorphisms}}]
\label{prop:BDR}
Let \( M \) be an orientable infinite-type 2-manifold, and let \( G \) be a large subgroup of \( \mcg(M) \).
A mapping class in \( G \) has finite-type support if and only if its conjugacy class in \( G \) is countable.
\qed
\end{Prop}

The statement in \cite{BavardIsomorphisms} is about finite-index subgroups of \( \mcg(M) \) and \( \pmcg(M) \), but the proof works just as well for large subgroups.

\begin{Prop}
\label{prop:theproof}
Let \( M \) be an infinite-genus 2-manifold with no planar ends, let \( M' \) be an orientable 2-manifold, and let \( G \) and \( G' \) be large subgroups of \( \mcg(M) \) and \( \mcg(M') \), respectively. 
If \( \varphi \co G \to G' \) is a continuous epimorphism such that \( \Gamma_c(M) \) is not contained in the kernel of \( \varphi \), then $\vp$ is induced by a homeomorphism. 
\end{Prop}

\begin{proof}
Let \( \vp \co G \to G' \) be a continuous epimorphism. 
An epimorphism maps conjugacy classes in the domain onto conjugacy classes in the codomain. 
In particular, if \( g \in G \) has a countable conjugacy class, then so does \( \vp(g) \).
Therefore, by Proposition~\ref{prop:BDR}, \( \vp(\Gamma_c(M)) < \Gamma_c(M') \), and by continuity, \( \vp(\Gamma(M)) < \Gamma(M') \). 
We can then restrict \( \vp \) to \( \Gamma(M) \) to get a continuous compact homomorphism \( \psi \co \Gamma(M) \to \Gamma(M') \). 
We claim that \( \psi \) is irreducible.

Let \( a \) and \( b \) be simple closed curves in \( M \) satisfying \( i(a,b) = 1 \). 
By Lemma~\ref{lem:multitwist}, \( \psi(t_a) \) and \( \psi(t_b) \) are multitwists; let \( A \) (resp., \( B \)) be the multicurve such that \( \psi(t_a) \) (resp., \( \psi(t_b) \)) twists about.  
Note that \( \psi(t_a) \) and \( \psi(t_b) \) are braided multitwists, and hence by Theorem~\ref{thm:depool}, either there exist \( a' \in A \) and \( b' \in B \) such that \( i(a',b') = 1 \) or \( \psi(t_a) = \psi(t_b) \).
If the latter case holds, then Lemma~\ref{lem:cyclic} implies that \( \Gamma_c(M) < \ker \psi \), which is not the case; hence, the former case holds, implying that \( A \) contains a non-separating simple closed curve. 

As \( \Gamma(M) \) is normal in \( G \) and as \( \vp \) is surjective, \( \vp(\Gamma(M)) \) is normal in \( G' \); in particular, \( \psi(\Gamma(M)) \) is normal in \( \Gamma(M') \). 
Let \( \alpha \) be a curve or  arc in \( M' \).
Let \( A \) be as above.
As \( A \) contains a non-separating curve, there exists \( h \in G' \) such that \( h(A) \) has nontrivial geometric intersection with \( \alpha \). 
It follows that \( (h\psi(t_a)h^{-1})(\alpha) \neq \alpha \), and as \( h\psi(t_a)h^{-1} \in \psi(\Gamma(M)) \), we can conclude that \( \psi \) is irreducible. 

We have established that \( \psi \co \Gamma(M) \to \Gamma(M') \) is a continuous compact irreducible homomorphism; hence, by Theorem~\ref{thm:compactly-supported}, there exists a homeomorphism \( h \co M \to M' \) such that \( h_* = \psi \), where $h_*$ is the map from  $\Gamma(M)$ induced by $h$. 
It is left to check that   that $\vp$ is also induced by $h$, or that, abusing notation, $h_*=\vp$. 
Let \( \vp_0\) be given by \( \vp_0 = h_*^{-1} \circ \vp \). 
Observe that for every simple closed curve \( c \) in \( M \), \( \vp(t_c) = \psi(t_c) \) by definition, and so \( \vp_0(t_c) = t_c \).
Let \( g \in G \), and let \( c \) be a simple closed curve in \( M \).
Then
\begin{align*}
t_{\vp_0(g)(c)} 	&= \vp_0(g) t_c \vp_0(g)^{-1} \\
			&= \vp_0(g) \vp_0(t_c) \vp_0(g)^{-1} \\
			&= \vp_0(g t_c g^{-1}) \\
			& = \vp_0(t_{g(c)}) \\
			&= t_{g(c)}.
\end{align*}
Therefore, \( \vp_0(g)(c) = g(c) \) for every simple closed curve \( c \) in \( M \).
By Proposition~\ref{prop:alexander-method}, \( \vp_0(g) = g \).
It follows that \( \vp_0 = h_*^{-1} \circ \vp \) is the identity isomorphism, implying \( \vp = h_* \).
Therefore, \( \vp \) is induced by a homeomorphism.
\end{proof}

As we will see below, Proposition~\ref{prop:theproof} will imply that either \( \vp \) fails to be surjective or \( G' \) contains a nontrivial normal abelian subgroup; the next lemma says the latter case is impossible. 

\begin{Lem}
\label{lem:abelian}
Let \( M \) be an infinite-type 2-manifold, and let \( G \) be a large subgroup of \( \mcg(M) \).
The trivial subgroup is the only normal abelian subgroup of \( G \).  
\end{Lem}

\begin{proof}
Let \( H \) be a nontrivial normal subgroup of \( G \). 
If \( h \in H \) is nontrivial, then there exists a curve \( a \) such that \( h(a) \neq a \). 
Let \( g_1 = [t_a, h] \), so
\[ g_1 = t_a \, h \, t_a^{-1} \, h^{-1} =  (t_a \, h \, t_a^{-1}) \, h^{-1} = t_a \, (h \,  t_a^{-1} \, h^{-1}) = t_a \, t_{h(a)}^{-1}. \]
The second equality implies that \( g_1 \in H \), and the last equality guarantees that \( g_1 \) is nontrivial, as \( h(a) \neq a \).
Let \( S_1 \) be the surface filled by a realization of \( \{ a, h(a) \} \); note that \( S_1 \) has finite type and contains the support of a representative of \( g_1 \). 
Thurston's construction (see \cite[Theorem~14.1]{Primer}) implies that \( g_1 \) is the image of a pseudo-Anosov element under the homomorphism \( \mcg(S_1, \partial S_1) \to G \) induced by the embedding \( S_1 \hookrightarrow M \) (if \( a \) and \( h(a) \) are disjoint, then \( S_1 \) is a union of disjoint annuli implying that \( g_1 \) is a multitwist, which we will also consider pseudo-Anosov in this disjoint union of annuli). 
It follows that if a curve \( c \) has nontrivial geometric intersection with \( \partial S_1 \), then \( g_1(c) \neq c \); in particular, \( g_1(c) = c \) if and only if \( i(c, \partial S) = 0 \). 

Let \( d \) be the isotopy class of a component of \( \partial S_1 \), and choose a curve \( c \) such that \( i(c,d)>0 \). 
Set \(  g_2 := t_c \, g_1 \, t_c^{-1} \in H \), and let \( S_2 \) be the image of \( S_1 \) under a representative of \( t_c \). 
Then \( g_2 \) is the image of a pseudo-Anosov element of \( \mcg(S_2, \partial S_2) \) under the homomorphism into \( \mcg(M) \) and \( i(d, \partial S_2) \neq 0 \). 
It follows that \( g_2(d) \neq d \) and \( g_1(d) = d \).
Therefore, \( g_1 \) and \( g_2 \) do not commute; indeed, if they commuted, then \( g_2 \) would preserve the fix set of \( g_1 \), which  is not the case.  
We can now conclude that \( H \) is not abelian, and the result follows, as \( H \) was an arbitrary nontrivial normal subgroup of \( G \). 
\end{proof}

We can now proof Theorem~\ref{thm:main}.

\begin{proof}[Proof of Theorem~\ref{thm:main}]
Let \( \varphi \co G \to G' \) be a continuous epimorphism. 
If \( \Gamma_c(M) \) is contained in the kernel of \( \vp \), then by continuity, \( \Gamma(M) \) is in the kernel of \( \vp \).
In this case,  \( \vp \) factors through \( G/\Gamma(M) \). 
By \cite{AramayonaFirst}, \( \pmcg(M)/ \Gamma(M) \) is abelian, so the image of \( (G \cap \pmcg(M)) / \Gamma(M) \) is an abelian normal subgroup of \( G' \) and hence is trivial by Lemma~\ref{lem:abelian}. 
But, as \( G \) is mostly pure, it follows that the image of \( \vp \) is countable, and hence cannot surject onto \( G' \), a contradiction.
Therefore, \( \Gamma_c(M) \) is not contained in the kernel of \( \varphi \).
By Proposition~\ref{prop:theproof}, \( \varphi \) is induced by a homeomorphism. 
\end{proof}

\subsection{Moving to the full mapping class group}
\label{sec:full}

If \( M \) has infinitely many ends, then \( \mcg(M) \) is not mostly pure, and hence, Theorem~\ref{thm:main} does not apply. 
The place where the proof is incomplete in this case is in establishing that \( \Gamma_c(M) \) is not in the kernel of the given epimorphism. 
Here we provide a strategy for dealing with this that works in some cases, but it is not clear if it can be adapted to work in all cases. 
To exhibit the technique, we extend Theorem~\ref{thm:main} to \( \mcg(M) \) when \( M \) is perfectly self-similar. 

\begin{Thm}
\label{thm:main2}
Let \( M \) and \( M' \) be orientable 2-manifolds, and suppose \( M \) is infinite genus with no planar ends. 
If \( M \) is perfectly self-similar and \( G' \) is a large subgroup of \( \mcg(M') \), then every continuous epimorphism from \( \mcg(M) \) to \( G' \) is induced by a homeomorphism. 
\end{Thm}

The main tool in proving Theorem~\ref{thm:main2} is a result of Afton--Calegari--Chen--Lyman.

\begin{Prop}[\!{\cite[Proposition~9]{AftonNielsen}}]
\label{prop:afton}
If \( M \) is an orientable 2-manifold, then there exists an open neighborhood of the identity in \( \mcg(M) \) such that the identity is the  only torsion element in the neighborhood.
\qed
\end{Prop}

The idea behind the proof of Theorem~\ref{thm:main2} is that if \( \Gamma_c(M) \) is in the kernel of an epimorphism, then there must be torsion elements arbitrarily close to the identity that normally generate \( G' \), contradicting Proposition~\ref{prop:afton}. 
The existence of these generators come from the following lemma. 

\begin{Lem}
\label{lem:involutions}
If \( E \) is the end space of a perfectly self-similar 2-manifold, then every neighborhood of the identity in \( \Homeo(E) \) contains an involution that normally generates \( \Homeo(E) \). 
\end{Lem}

\begin{proof}
By \cite[Proposition~5.4.6]{VlamisHomeomorphism}, there exists \( x \in E \) and a sequence \( \{ E_n \}_{n\in\bn} \) of pairwise-homeomorphic and disjoint clopen subsets of \( E \) such that \( E \ssm \{x\} = \bigcup_{n\in\bn} E_n \). 
From this decomposition, for each \( n \in \bn \), we can find a involution \( \tau_n \co E \to E \) such that \( \tau_n(E_n) = E_{n+1} \) and \( \tau_n \) restricts to the identity on \( E_k \) for \( k \neq n, n+1 \). 
Observe that \( \lim_{n\to \infty} \tau_n \) is the identity. 
Moreover, \cite[Corollary~5.2.3]{VlamisHomeomorphism} implies that \( \tau_n \) is a normal generator for \( \Homeo(E) \). 
\end{proof}

\begin{proof}[Proof of Theorem~\ref{thm:main2}]
Let \( \varphi \co \mcg(M) \to G' \) be a continuous epimorphism, and suppose that \( \Gamma_c(M) \) is contained in \( \ker \vp \). 
Arguing as in the proof of Theorem~\ref{thm:main}, it follows that \( \vp \) induces an epimorphism \( \mcg(M)/\pmcg(M) \to G' \).
Richards's proof of the classification of surfaces \cite{RichardsClassification}, together with the fact that all the ends are nonplanar, implies that \( \mcg(M)/\pmcg(M) \) is isomorphic to \( \Homeo(\Ends(M)) \), where \( \Ends(M) \) is the end space of \( M \). 
Let \( \bar\vp \) denote the induced epimorphism \( \Homeo(\Ends(M)) \to G' \). 
By Lemma~\ref{lem:involutions}, there exists a sequence of involutions \( \{\tau_n\}_{n\in\bn} \) in \( \Homeo(\Ends(M)) \) that limit to the identity and such that each involution normally generates \( \Homeo(\Ends(M)) \).
But Proposition~\ref{prop:afton} implies that \( \bar\vp(\tau_n) \) must be the identity for all large \( n \).
As the \( \tau_n \) are normal generators, \( \bar\vp \) is trivial, a contradiction.
Therefore, \( \Gamma_c(M) \) is not in the kernel of \( \vp \), and hence \( \vp \) is induced by a homeomorphism by Proposition~\ref{prop:theproof}.
\end{proof}

The argument in the proof of Theorem~\ref{thm:main2} works for a larger class of 2-manifolds; however, it is not clear how to naturally define the class.
For example, we can extend to another class of self-similar 2-manifolds using \cite{BhatAlgebraic}.
In particular, replacing the use of \cite[Proposition~5.4.6]{VlamisHomeomorphism} and \cite[Corollary~5.2.3]{VlamisHomeomorphism} in Lemma~\ref{lem:involutions} with \cite[Proposition~2.22]{BhatAlgebraic} and \cite[Theorem~3.12]{BhatAlgebraic}, respectively, then the following theorem has the same proof as Theorem~\ref{thm:main2}.

\begin{Thm}
Let \( M \) and \( M' \) be orientable 2-manifolds.
Suppose \( M \) is infinite genus with no planar ends and that the end space of \( M \) is countable with Cantor--Bendixson degree a successor ordinal and with Cantor--Bendixson rank one (i.e., it is homeomorphic to an ordinal of the form \( \omega^{\alpha+1} +1 \) equipped with its order topology, where \( \omega \) is the first countable ordinal and \( \alpha \) is a countable ordinal). 
If \( G' \) is a large subgroup of \( \mcg(M') \), then every continuous epimorphism from \( \mcg(M) \) to \( G' \) is induced by a homeomorphism.
\qed
\end{Thm}

We do not know if the proof of Theorem~\ref{thm:main2} can be extended to cover all infinite-genus self-similar 2-manifolds with no planar ends; however, we know the strategy cannot work for all infinite-genus 2-manifolds.
For instance, consider the infinite-genus 2-manifold \( M \) with no planar ends and whose end space \( E \) is homeomorphic to \( \omega\cdot 2 + 1 \) (or equivalently, the two-point compactification \( \widehat \bz = \bz \cup \{\pm \infty \} \) of \( \bz \)).
 Then the closure of the finitely supported homeomorphisms in \( \Homeo(E) \) of \( E \), call it \( G \), is an open neighborhood of the identity and a normal subgroup of \( \Homeo(E) \) such that \( \Homeo(E) / G \) is isomorphic to \( (\bz/2\bz)^2 \) (this follows from the results in \cite[Section~4]{BhatAlgebraic}).
In particular, \( G \) is a neighborhood of the identity that cannot contain a normal generating set of \( \Homeo(E) \), implying the techniques in the proof of Theorem~\ref{thm:main2} cannot be applied to \( M \) (though, we note that this discussion provides enough information to make an argument that allows one to extend the conclusion of Theorem~\ref{thm:main2} to \( \mcg(M) \)). 

%The argument in the proof of Theorem~\ref{thm:main2} works for a larger class of 2-manifolds; however, it is not clear how to naturally define the class.
%But, for example, the proof also implies that if \( M \) is an orientable infinite-genus 2-manifold with no planar ends and with \( \Ends(M) \) homeomorphic to \( \omega+1 \), then Theorem~\ref{thm:main2} holds for \( M \). 
%In this case, \( \Homeo(\Ends(M)) \) is topologically isomorphic to \( \mathrm{Sym}(\bn) \), and so contains a dense subgroup generated by involutions arbitrary close to the identity (namely, the subgroup consisting of finitely supported permutations). 

%----------
% Relaxing continuity
%----------

\section{Relaxing continuity}

In this final section, we weaken the continuity assumption in Theorem~\ref{thm:main}.
The idea is to replace continuity with a version of ``algebraic continuity'' that preserves certain limits that can be described in terms of infinite multiplication.
We will work with a version of \emph{infinitely multiplicative} homomorphisms---as introduced by Cannon--Conner \cite{CannonCombinatorial}---suited to our purposes. 

Let us begin by expanding our discussion of infinite products from Section~\ref{sec:prelim}.
A sequence of mapping classes \( \{ f_n \}_{n\in\bn} \) in \( \mcg(M) \) is \emph{locally finitely supported} if \( \{ n \in \bn : f_n(a) \neq a \} \) is finite for every curve \( a \) in \( M \).
If \( \{ f_n\}_{n\in\bn} \) is locally finitely supported, then we can define \( f = \prod_{n\in\bn} f_n \in \mcg(M) \) as follows: given a curve \( a \) in \( M \) there exists \( N_a \in \bn \) such that \( f_n(a) = a \) for all \( n > N_a \), allowing us to define \( f(a) = (f_1 \, f_2 \, \cdots \, f_{N_a})(a) \).
Note that, in the compact-open topology, \( f = \lim \prod_{i=1}^n f_i \), where we write the product left to right, i.e., \( \prod_{i=1}^n f_i = f_1 \, f_2 \, \cdots \, f_n \).

\begin{Def}
Let \( M \) and \( M' \) be 2-manifolds with \( M \) nonplanar, and let \( G \) and \( G' \) be subgroups of \( \mcg(M) \) and \( \mcg(M') \), respectively. 
A homomorphism \( \vp \co G \to G' \) is \emph{infinitely multiplicative on twists} if, for any locally finitely supported sequence of non-separating Dehn twists \( \{ t_n \}_{n \in \bn} \), the infinite product \( \prod_{n \in \bn} \vp(t_n) \) exists. 
\end{Def}

We only need one lemma to replace the continuity assumption in Theorem~\ref{thm:main} with requiring the epimorphisms to be infinitely multiplicative on twists. 

\begin{Lem}
\label{lem:infinite-products}
Let \( M \) be an orientable infinite-genus 2-manifold in which every planar end is isolated. 
If \( f \in \Gamma(M) \ssm \Gamma_c(M) \), then there exists a sequence \( \{t_n\}_{n\in\bn} \) of non-separating Dehn twists such that \(  \prod_{n\in\bn} t_n \) exists and is equal to \( f \).
\end{Lem}

\begin{proof}
Let \( \{ \Sigma_n \}_{n\in\bn} \) be an exhaustion of \( M \) by finite-type subsurfaces.
Let \( U_n \subset \mcg(M) \) be defined as follows:  \( f \in U_n \) if and only if \( f(a) = a \) for every curve \( a \) with representative contained in \( \Sigma_n \). 
Then \( \{ U_n\}_{n\in\bn} \) is a neighborhood basis for the identity in \( \mcg(M) \). 

As \( \Gamma_c(M) \) is generated by non-separating Dehn twists and is dense in \( \Gamma(M) \), there exists \( f_1 \in \Gamma_c(M) \) such that \( f_1 \in f \, U_1 \) and such that \( f_1 \) is a finite product of non-separating Dehn twists. 
Therefore, \( f_1^{-1} \, f \in U_1 \). 
Arguing as above, there exists \( f_2 \in f_1^{-1} \, f \, U_2 \) such that \( f_2 \) is a finite product of non-separating Dehn twists.  
Moreover, \( f_2 \in U_1 \).
Proceeding recursively, we obtain a sequence of mapping class \( \{ f_n \}_{n\in\bn} \) such that
\begin{itemize}
\item \( f_{n+1} \in U_n \),
\item \( f_n^{-1} \, \cdots \, f_1^{-1} f \in U_n  \), and
\item \( f_n \) is a finite product of non-separating Dehn twists. 
\end{itemize}
It follows that \( \{ f_n \} \) is locally finitely supported,  \( \lim f_n = f \), and \( f = \prod_{n \in \bn } t_n \) for some locally finitely supported sequence of non-separating Dehn twists \( \{ t_n \}_{n\in\bn} \).  
\end{proof}

Tracing through the various results in the prior sections, one readily checks that every instance of continuity can be replaced with infinite multiplicativity on twists using Lemma~\ref{lem:infinite-products}, which yields a more algebraic version of the main theorem.

\begin{Thm}
\label{thm:main3}
Let \( M \) and \( M' \) be orientable 2-manifolds and suppose \( M \) is infinite genus and has no planar ends.
If \( G \) is a mostly pure large subgroup of \( \mcg(M) \) and \( G' < \mcg(M') \) is large, then every epimorphism \( G \to G' \) that is infinitely multiplicative on twists is induced by a homeomorphism.
\qed
\end{Thm}

\bibliographystyle{amsplain}
\bibliography{references}

\end{document}